\documentclass[pdflatex,sn-mathphys-num]{sn-jnl}

\usepackage{orcidlink}

\usepackage[english]{babel}
\usepackage[utf8]{inputenc}
\usepackage[T1]{fontenc}
\usepackage{setspace}
\usepackage{listings}
\usepackage{amsmath}
\usepackage{amsthm}
\usepackage{amssymb}
\usepackage{fancybox}
\usepackage[nooneline,labelfont=bf, font=small, format=plain,justification=raggedright,singlelinecheck=off]{caption}
\usepackage{enumerate}

\usepackage{pdfpages}
\usepackage{physics}
\usepackage{mathtools}
\usepackage[vlined, ruled, algo2e, linesnumbered]{algorithm2e}
\usepackage{booktabs}
\usepackage{tabularx}
\usepackage{multirow}

\usepackage{sidecap}

\usepackage{todonotes}
\usepackage{url}
\usepackage{doi}

\theoremstyle{lemma}
\theoremstyle{theorem}
\theoremstyle{corollary}
\theoremstyle{definition}
\theoremstyle{definition}

\newtheorem{theorem}{Theorem}[section]

\newtheorem{definition}[theorem]{Definition}
\newtheorem{remark}[theorem]{Remark}
\newtheorem{lemma}[theorem]{Lemma}

\newtheorem{problem}[theorem]{Problem}

\newtheorem{procedure__}[theorem]{Procedure}

\newtheorem*{problem*}{Problem}

\counterwithout{footnote}{section}

\usepackage{acronym}
\usepackage[nonumberlist]{glossaries}

\DeclareMathOperator*{\argmin}{argmin}

\newacronym[plural=FOMs]{FOM}{FOM}{Full Order Model}
\newacronym{AGC}{AGC}{Approximate generalized Cauchy point}
\newacronym{ROM}{ROM}{Reduced Order Model}
\newacronym{MOR}{MOR}{Model Order Reduction}
\newacronym{ML}{ML}{Machine Learning}
\newacronym{TR}{TR}{Trust-Region}
\newacronym{RB}{RB}{Reduced Basis}
\newacronym[plural=PINNs]{PINN}{PINN}{Physics-informed Neural Network}
\newacronym{DNN}{DNN}{Deep Neural Network}
\newacronym[plural=PDEs]{PDE}{PDE}{Partial Differential Equation}
\newacronym[plural=pPDEs]{pPDE}{pPDE}{parameterized PDE}
\newacronym{VKOGA}{VKOGA}{Vectorial Kernel Orthogonal Greedy Algorithm}
\newacronym{DoF}{DoF}{Degree of Freedom}
\newacronym{HaPOD}{HaPOD}{Hierarchical approximate POD}
\newacronym{FE}{FE}{Finite Element}
\newacronym{BFGS}{BFGS}{Broyden, Fletcher, Goldfarb, Shanno}
\glsdisablehyper

\begin{document}

\title{Multi-fidelity Learning of Reduced Order Models for Parabolic PDE Constrained Optimization}

\author*[1]{\fnm{Benedikt} \sur{Klein} \orcidlink{0009-0004-1909-3574}}\email{benedikt.klein@uni-muenster.de}
\author[2]{\fnm{Mario} \sur{Ohlberger} \orcidlink{0000-0002-6260-3574}}\email{mario.ohlberger@uni-muenster.de}

\affil[1,2]{\orgdiv{Mathematics Münster}, \orgname{University of Münster}, \orgaddress{\street{Einsteinstrasse 62}, \city{Münster}, \postcode{48149}, \country{Germany}}}


\abstract{This article builds on the recently proposed RB-ML-ROM approach for parameterized parabolic PDEs and proposes a novel hierarchical Trust Region algorithm for solving parabolic PDE constrained optimization problems. 
Instead of using a traditional offline/online splitting approach for model order reduction, we adopt an active learning or enrichment strategy to construct a multi-fidelity hierarchy of reduced order models on-the-fly during the outer optimization loop.
The multi-fidelity surrogate model consists of a full order model, a reduced order model and a machine learning model. 
The proposed hierarchical framework adaptively updates its hierarchy when querying parameters, utilizing a rigorous a posteriori error estimator in an error aware trust region framework. Numerical experiments are given to demonstrate the efficiency of the proposed approach.}

\keywords{Reduced order models, multi-fidelity learning, parabolic PDE constrained optimization, trust region algorithm}

\pacs[MSC Classification]{49M20, 49K20, 35J20, 65N30, 90C06}

\maketitle

\section{Introduction}
Optimization problems governed by \glspl{PDE} play a crucial role in various fields, including physics, engineering, and economics, as they enable the modeling, solution and optimization of complex systems 
that involve spatially distributed phenomena. 
The ability to efficiently solve such optimization problems has far-reaching implications for applications like optimal control, 
inverse problems, and design optimization. Typical algorithms require numerous evaluations of the underlying \gls{PDE} for 
various parameters that are selected in an outer descent-type optimization loop.
In this contribution we are concerned with a novel multi-fidelity learning approach to speed up such algorithms for 
parameter optimization or optimal control problems with parabolic \gls{PDE} constraints. 
The approach builds on a hierarchy of approximate models with increasing accuracy and decreasing efficiency, 
consisting of a machine learning surrogate, a reduced basis model and a full order finite element model.
Such an approach has recently been proposed as a certified “on demand” \acrshort{RB}-\acrshort{ML}-\acrshort{ROM} learning approach to approximate parametrized input-output maps \cite{haasdonk2022new}.
In order to efficiently handle optimization problems, we integrated this approach into the error aware trust region framework \cite{yue2013accelerating} that has recently been advanced theoretically and numerically in \cite{Keil_2021,banholzer2020adaptiveprojectednewtonnonconforming,keil2024relaxed}.
We now give a more detailed overview to the various techniques involved in our mutli-fidelity approach. \\

\textbf{Model order reduction for PDE constrained optimization.}
\gls{MOR} methods have gained significant attention in recent years for efficiently solving parameterized \glspl{PDE} \cite{MR3701994, MR3672144}. These methods aim to reduce the computational complexity by replacing high-dimensional \glspl{FOM} with low-dimensional surrogates. 

A promising approach is the usage of \gls{RB} methods, which rely on approximating the solution manifold of \glspl{pPDE} by low-dimensional linear approximation spaces. We particularly refer to the POD-Greedy method \cite{Haasdonk2008,Haasdonk13} as well as the \gls{HaPOD} \cite{himpe2018hierarchical} for parameterized parabolic \glspl{PDE}. There is a large amount of literature using such reduced models for \gls{PDE} constrained optimization, cf.\cite{KTV13,GK2011,NRMQ2013}. 
For an introduction to \gls{MOR} methods for optimal control problems we refer, e.g., to \cite{MR4628189}. 
More recently, \gls{RB} methods have been advised with a progressive construction of \glspl{ROM}
\cite{ZF2015,GHH2016,MR3823611}
which has lead to a serious of works on error aware trust region reduced basis methods.\\

\textbf{Error aware Trust-Region -- Reduced Basis methods.}
\Gls{TR} approaches are a class of optimization methods that build on the iterative 
solution of appropriately chosen locally accurate surrogate models. 
The main idea is to solve optimization sub-problems only in a local
area of the parameter domain which resolves the burden of constructing a global \Gls{RB} space.
\Gls{TR} optimization methods have been successfully combined with \Gls{MOR} to leverage locally accurate surrogate models, ensuring global convergence while reducing the computational cost \cite{yue2013accelerating,RoggTV17}. 
Significant further development and analysis of such approaches were presented in 
\cite{qian2017certified,Keil_2021,banholzer2020adaptiveprojectednewtonnonconforming,keil2024relaxed}.
We also refer to a related approach with application to shape optimization problems \cite{MR4569234}.\\

\textbf{Data based surrogate modeling and machine learning approaches.}
In parallel to model-based reduction methods, purely data-based
approaches for machine learning of surrogate models have increasingly been developed
and mathematically investigated \cite{hesthaven2018non,Kutyniok22,fresca2022pod,ManzoniHesthaven23,FrescaManzoni24}.
Our methodology exploits the strengths of both model-based and data-driven approaches, 
enabling the construction of certified, adaptive, and efficient surrogate models. 
The general idea of a  full order, reduced order and machine learning model
pipeline has first been introduced in \cite{MR4406092} and subsequently been 
put in a certified \Gls{RB}-\acrshort{ML}-\Gls{ROM} framework in \cite{haasdonk2022new}. 
Despite deep neural networks, also kernel methods and novel deep kernel approaches 
have been investigated and compared \cite{wenzel2023application}. 
For further learning based approaches for related optimization problems
we refer to \cite{lye2021iterative, keil2022adaptive, xu2020physicsconstrainedlearningdatadriven,MR4848008}
We also note that these approaches are complementary to unsupervised learning approaches, such as physics-informed neural networks  \cite{raissi2019physics} or, e.g., deep Ritz methods \cite{MR3767958}.\\

\textbf{Main results.}
This work demonstrates the potential of combining the above mentioned approaches to 
tackle parameter optimization with parabolic \gls{PDE} constraints or corresponding optimal control problems. 
We thereby integrate certified and adaptive model order reduction with machine learning techniques in a trust region 
optimization framework. 
By doing so in a multi-fidelity fashion, we overcome the limitations of traditional global model reduction methods and purely machine learning based approaches that typically come without error certification.
Particular novelties of this contribution include

\begin{itemize}
\item a posteriori error estimates for the approximate objective function and its derivatives,
\item a concept for trust region multi-fidelity optimization,
\item the integration of the \gls{RB}-\acrshort{ML}-\gls{ROM} surrogate modeling approach with the error aware \gls{TR} optimization method, 
\item an efficient implementation of the resulting multi-fidelity \gls{TR}-\gls{RB}-\acrshort{ML}-Opt approach based on the model reduction software framework pyMOR \cite{Milk_2016},
\item a numerical proof of concept that demonstrate that our new \gls{TR}-\gls{RB}-\acrshort{ML}-Opt method outperforms
existing model reduction approaches. 
\end{itemize}

\textbf{Organization of the article.}
In Section \ref{sec:problem} we introduce the parabolic \gls{PDE} constrained optimization problem and 
discuss the adjoint problem resulting from the optimality conditions. A suitable reduced basis approximation and corresponding a posteriori error estimates are discussed in Section \ref{sec:MOR}. Section \ref{sec:ML_surrogates} then introduces 
machine learning surrogates and adaptations that are necessary with respect to a posteriori error estimation. 
Finally, the trust region multi-fidelity optimization framework is introduced in Section \ref{sec:TR}. Numerical experiments in Section \ref{sec:num_exps} demonstrate the performance of the resulting approach in comparison to previous \gls{TR}-\gls{RB} approaches.

\section{Problem Formulation}
\label{sec:problem}

Let $T > 0$ be a fixed end-time and $\Omega \subset \mathbb{R}^d$ a bounded domain with a Lipschitz continuous boundary $\partial \Omega$. We define the Hilbert spaces $\mathbb{D}$ and $V$, equipped with inner products $\langle\cdot, \cdot\rangle_{\mathbb{D}}$ and $\langle \cdot, \cdot \rangle_{V}$, along with their induced norms $\| \cdot \|_{\mathbb{D}} := \sqrt{\langle \cdot, \cdot \rangle_{\mathbb{D}}}$ and $\| \cdot \|_{V} := \sqrt{\langle \cdot, \cdot \rangle_{V}}$. These spaces satisfy the embedding hierarchy ${H^1_0(\Omega) \subset V \subset H^1(\Omega) \subset L^2(\Omega) \subset V'}$. 

Additionally, let $\mathcal{P} \subset \mathbb{R}^P$ be a simple bounded parameter domain, defined as $\mathcal{P} := \{x \in \mathbb{R}^P \mid L_j \leq x_j \leq U_j, \quad j = 1, \dots, P\}$, where $-\infty < L_j < U_j < \infty$. Our goal is to determine a parameter $\mu \in \mathcal{P}$ that minimizes the time-discretized least-squares objective functional relative to a desired state $g_\text{ref} \in L^2(0,T; \mathbb{D})$: 
\begin{equation*}
J(u(\mu); \mu) := \Delta t \sum_{k = 1}^K \|F(u^k(\mu)) - g^k_\text{ref} \|_{\mathbb{D}}^2 + \lambda \mathcal{R}(\mu),
\end{equation*}
where $g^k_\text{ref} := g_\text{ref}(t^k)$ are evaluations of $g_\text{ref}$ at a uniform temporal grid ${\mathcal{G}^\text{pr}_{\Delta t} := \{0 = t^0 < \dots < t^K = T\}}$ with time step $\Delta t := t^{i+1} - t^i$. The state trajectory $u(\mu) := (u^k(\mu))_{k \in \{0, \dots, K\}}$ satisfies the time-discretized parabolic PDE: 
\begin{align}
\begin{split}
\label{eq:primal_analytic_PDE}
\frac{1}{\Delta t}(u^k(\mu)-u^{k-1}(\mu), v )_{L^2(\Omega)} + a(u^k(\mu), v ; \mu) &= b(t^k)f(v; \mu), \quad
u^0(\mu) = 0,
\end{split}
\end{align}
for all $k \in \mathbb{K} := \{1, \dots, K\}$ and $v \in V$. Here, for each $\mu \in \mathcal{P}$, the operator ${f(\cdot\,; \mu) : V \rightarrow \mathbb{R}}$ is a $V$-continuous, twice continuously differentiable linear functional with derivatives $\partial_{\mu_i} f(\cdot\,; \mu) \in V '$ for $i \in \{  1, \dots, P \}$. The bilinear form ${a(\cdot, \cdot\,; \mu) : V \times V \rightarrow \mathbb{R}}$ is $V$-continuous, coercive, symmetric, and twice continuously differentiable in $\mu$, with derivatives $a_{\mu_i} := \partial_{\mu_i} a(\cdot, \cdot\,; \mu)$. The mapping $F: V \rightarrow \mathbb{D}$ computes a quantity of interest for states in $V$, while $b : [0,T] \rightarrow \mathbb{R}$ represents a time-dependent forcing term. Regularization is ensured by a twice continuously differentiable function $\mathcal{R} : \mathcal{P} \rightarrow \mathbb{R}$ with weight $\lambda \in \mathbb{R}_{\geq 0}$. Henceforth, \eqref{eq:primal_analytic_PDE} will be referred to as the \textit{primal problem}.

For the construction of an a posteriori error estimator, we define the inner product on $V$ using the $a$-energy product at a fixed $\bar{\mu} \in \mathcal{P}$, i.e., $\langle \cdot, \cdot \rangle_V := a(\cdot, \cdot\,; \bar{\mu})$. We assume that for any given $\mu \in \mathcal{P}$, the coercivity and continuity constants
\begin{align*}
0 < \gamma_a(\mu) := \sup_{w \in V \setminus \{ 0 \}} \sup_{v \in V \setminus \{ 0 \}} & \frac{a(w, v ; \mu)}{\|w\|_V \|v\|_V} < \infty, \\
0 < \gamma_{a_{\mu_i}}(\mu) := \sup_{w \in V \setminus \{ 0 \}} \sup_{v \in V \setminus \{ 0 \}} & \frac{a_{\mu_i}(w, v ; \mu)}{\|w\|_V \|v\|_V} < \infty \text{ and } \\
0 < \alpha_0^a \leq \alpha(\mu) := \inf_{v \in V \setminus \{ 0 \}} & \frac{a(v, v ; \mu)}{\|v\|^2_V},
\end{align*}
can be bounded from below, respectively from above. I.e. we assume that $\alpha_\text{LB}(\mu)$ and $\gamma_{a_{\mu_i}}^{UB}(\mu)$ exist such that
\begin{gather*}
0 < \alpha_0^a \leq \alpha_\text{LB}(\mu) \leq \alpha(\mu) \text{ and }
\gamma_{a_{\mu_i}}^{\text{UB}}(\mu) \geq \gamma_{a_{\mu_i}}(\mu).
\end{gather*}
Furthermore, we assume that $a$ and $f$ are \textit{affinely decomposable}, meaning that
\begin{equation*}
a(w, v; \mu) = \sum_{q =  1}^{Q_a} \Theta^q_a(\mu) a^q(w,v), \quad f(v; \mu) = \sum_{\tilde{q} =  1}^{Q_f} \Theta^{\tilde{q}}_f(\mu) f^{\tilde{q}}(v),
\end{equation*}
where the functions $\Theta^q_a, \Theta^{\tilde{q}}_f: \mathcal{P} \rightarrow \mathbb{R}$ are twice continuously differentiable, while $a^q: V \times V \rightarrow \mathbb{R}$ and $f^{\tilde{q}}: V \rightarrow \mathbb{R}$ are parameter-independent continuous (bi)linear forms.
For notational convenience, we introduce:
\begin{equation*}
d(w, v) := \left(F(w),F(v) \right)_\mathbb{D}, \quad l^k(v) := -2 \left(F(v),g^k_\text{ref} \right)_\mathbb{D} \text{ for } v, w \in V,
\end{equation*}
with continuity constants $\gamma_d$ and $\gamma_l$. Throughout this work, the state trajectory $u(\mu)$ belongs to the Bochner-type Hilbert space
\begin{equation}
\label{eq:def_Q_space}
Q^\text{pr}_{\Delta t}(0,T; V) := \left\{ f: \mathcal{G}^\text{pr}_{\Delta t} \rightarrow V \mid f(t^{0}) = 0, \quad \|f\|_{Q^\text{pr}_{\Delta t}(0,T; V)} < \infty \right\},
\end{equation}
with inner product 
\begin{equation*}
\langle f, g \rangle_{Q^\text{pr}_{\Delta t}(0,T; V)} := \sum_{k = 1}^K \langle f(t^k), g(t^k) \rangle_{V},
\end{equation*}
and corresponding norm $\|\cdot\|_{Q^\text{pr}_{\Delta t}(0,T; V)}$. The optimization problem of interest is therefore:
\begin{problem}
Consider the minimization problem
\label{prob:t_dis_opt_prob}
\begin{gather*}
\argmin_{\mu \in \mathcal{P}} \mathcal{J}(\mu), \\
\begin{align*}
\mathcal{J}(\mu) & := J(u(\mu); \mu) = \Delta t \sum_{k = 1}^K \left[d(u^k(\mu), u^k(\mu)) + l^k(u^k(\mu)) +  \left(g^k_\text{ref}, g^k_\text{ref} \right)_\mathbb{D} \right] + \lambda \mathcal{R}(\mu),
\end{align*}
\end{gather*}
where  $u(\mu) \in Q^\text{pr}_{\Delta t}(0,T,V)$ satisfies \eqref{eq:primal_analytic_PDE}.
\end{problem}

\begin{remark}
Following the assumptions made above, for all $\mu \in \mathcal{P}$, a unique solution to \eqref{eq:primal_analytic_PDE} exists. Thus, the primal solution operator $A^{\text{pr}} : \mathcal{P} \rightarrow Q^{\text{pr}}_{\Delta t}(0,T;V);\,{\mu \mapsto u(\mu)}$ is well-defined. Moreover, it can be shown, cf. \cite{hinze2009optimization}, that $A^{\text{pr}}$ and thus $\mathcal{J}(\mu)$ are continuously differentiable with respect to $\mu$. The $i$-th partial derivative will be denoted by $d_{\mu_i}\mathcal{J}(\mu)$. Furthermore, the Fr\'{e}chet derivative is denoted by $\partial_\mu \mathcal{J}(\mu): \mathcal{P} \rightarrow \mathbb{R}$. Note that $d_{\mu_i} \mathcal{J}(\mu) = \partial_\mu \mathcal{J}(\mu)\cdot e_i$, where $e_i$ is the $i$-th canonical unit vector.
\end{remark}

\subsection{The Adjoint Problem}
To solve optimization problem \ref{prob:t_dis_opt_prob} using a gradient descent type method, efficient computation of the derivatives $d_{\mu_i} \mathcal{J}(\mu)$ is required. This can be achieved by an adjoint approach as outlined in \cite{hinze2009optimization, qian2017certified, Keil_2021}. For a fixed $\mu \in \mathcal{P}$, let ${u(\mu) \in Q^{\text{pr}}_{\Delta t}(0,T; V)}$ be the solution to \eqref{eq:primal_analytic_PDE}. The \textit{adjoint trajectory} associated to $u(\mu)$, denoted ${p(\mu) := (p^k(\mu))_{k \in \{1,\dots, K+1\}}}$, is given as the unique solution of
\begin{align}
\label{eq:dual_analytic_PDE}
\begin{split}
\frac{1}{\Delta t} (v, p^k(\mu)-p^{k+1}(\mu))_{L^2(\Omega)} + a(v, p^k(\mu); \mu) &= 2d(u^k(\mu), v) + l^k(v),\\
p^{K+1}(\mu) &= 0,
\end{split}
\end{align}
for all $k \in \mathbb{K}$ and $v \in V$. The Bochner-type Hilbert space $Q^{\text{ad}}_{\Delta t}(0,T; V)$ is defined analogously to \eqref{eq:def_Q_space},
\begin{equation*}
Q^{\text{ad}}_{\Delta t}(0,T; V) := \left\{ f: \mathcal{G}^{\text{ad}}_{\Delta t} \rightarrow V \mid f(t^{K+1}) = 0, \quad \|f\|_{Q^{\text{ad}}_{\Delta t}(0,T; V)} < \infty \right\},
\end{equation*}
with $\mathcal{G}^{\text{ad}}_{\Delta t} := \lbrace \Delta t =: t^1 < \dots < t^{K+1} := T + \Delta t \rbrace$. Following \cite{qian2017certified, hinze2009optimization}, the first-order derivative $d_{\mu_i} \mathcal{J}(\mu)$ for $i \in \{1, \dots, P\}$ can be computed as
\begin{equation}
\label{eq:nabla_J_tilde_op}
d_{\mu_i} \mathcal{J}(\mu) = \Delta t \sum_{k = 1}^K \left[b(t^k)f_{\mu_i}(p^k(\mu); \mu) - a_{\mu_i}(u^k(\mu), p^k(\mu); \mu)  \right] + \lambda d_{\mu_i} \mathcal{R}(\mu).
\end{equation}
This adjoint-based approach provides computational efficiency by solving \eqref{eq:dual_analytic_PDE} instead of evaluating each derivative separately.

\begin{remark}
The formulation \eqref{eq:nabla_J_tilde_op} is equivalent to $d_{\mu_i} \mathcal{J}(\mu) = d_{\mu_i} \mathcal{L}(u(\mu), \mu, p(\mu))$ (cf. \cite{Keil_2021}), where $\mathcal{L}$ is the Lagrangian for Problem \ref{prob:t_dis_opt_prob}, given by
\begin{align*}
\mathcal{L}(u(\mu), \mu, p(\mu)) = \mathcal{J}(\mu) + & \sum_{k = 1}^K b(t^k)f(p^k(\mu) ; \mu) - a(u^k(\mu), p^k(\mu) ; \mu) \\
- & \frac{1}{\Delta t} (u^k(\mu) - u^{k - 1}(\mu), p^k(\mu) )_{L^2(\Omega)}.
\end{align*}
\end{remark}

For solving \eqref{eq:primal_analytic_PDE} and its adjoint problem \eqref{eq:dual_analytic_PDE}, a conventional \gls{FE} discretization approach is used. Consider a finite-dimensional Hilbert space $V_h \subset V$, equipped with the inner product $\langle \cdot, \cdot \rangle_{V}$. By projecting the primal \eqref{eq:primal_analytic_PDE} and adjoint problems \eqref{eq:dual_analytic_PDE} into $V_h$ (Galerkin projection), we obtain well-defined approximations $u_h(\mu) \in Q^{\text{pr}}_{\Delta t}(0,T; V_h)$ and $p_h(\mu) \in Q^{\text{ad}}_{\Delta t}(0,T; V_h)$ for the exact solutions $u(\mu)$ and $p(\mu)$. The corresponding objective functional and gradient are given by
\begin{equation}
\label{eq:def_J_h}
\mathcal{J}_h(\mu) := J(u_h(\mu); \mu), \quad d_{\mu_i} \mathcal{J}_h(\mu) = d_{\mu_i} \mathcal{L}(u_h(\mu), \mu, p_h(\mu)).
\end{equation}
We refer to the computational model that determines, for a given parameter $\mu \in \mathcal{P}$, the high-fidelity objective functional $\mathcal{J}h(\mu)$ and its gradient ${\nabla\mu \mathcal{J}h(\mu) := (\partial{\mu_1} \mathcal{J}h(\mu), \dots, \partial{\mu_P} \mathcal{J}_h(\mu))}$ as the \acrfull{FOM}.
\section{Reduced basis method}
\label{sec:MOR}
Ensuring satisfactory accuracy of the \glspl{FOM} typically requires a high dimension of $V_h$, making their evaluation computationally expensive. This is especially challenging in optimization, due to the high number of parameter queries required. To mitigate the computational burdens of high-fidelity simulations, \Gls{RB} methods have gained considerable interest and demonstrated substantial reductions in computation time \cite{Keil_2021, qian2017certified, quarteroni2015reduced}. The properties of \Gls{RB} models make them also genuinely compatible with the trust-region method as will discussed in detail in Section \ref{sec:TR}. Moreover, can \Gls{RB} methods act as vital cornerstone for integrating machine learning-based surrogates in optimization problems, as a feasible option, as we will see in Section \ref{sec:ML_surrogates}.

In RB methods, the high-dimenstional space $V_h$ is replaced by a subspace ${V_{\text{RB}} \subset V_h}$, with a significantly lower dimension and the problems \eqref{eq:primal_analytic_PDE} and \eqref{eq:dual_analytic_PDE} are solved, after projecting them into this space, allowing for faster evaluations due to the reduced dimension. Suppose a $V_h$-orthonormal basis $\{\psi_i,\, i = 1, \dots, N_\text{RB} \}$ is given. This \textit{reduced basis}, define the reduced state space
\begin{equation}
\label{eq:def_RB_space}
V_\text{RB} := \text{span} \{\psi_i,\, i = 1, \dots, N_\text{RB} \} \subset V_h.
\end{equation}
The \Gls{RB} approximation $u_\text{RB}(\mu) \in Q^\text{pr}_{\Delta t}(0,T; V_\text{RB})$ to the primal high-fidelity trajectory $u_h(\mu)$  thus satisfies for all $k \in \mathbb{K}$ and $v \in V_\text{RB}$
\begin{align}
\label{eq:primal_ROM_PDE}
\begin{split}
\frac{1}{\Delta t}(u_\text{RB}^k(\mu)-u_\text{RB}^{k-1}(\mu), v )_{L^2(\Omega)} + a(u_\text{RB}^k(\mu), v ; \mu) &= b(t^k)f(v; \mu),  \quad
u^0_\text{RB}(\mu) = 0.
\end{split}
\end{align}
Similarly, the approximation of the adjoint solution $ p_\text{RB}(\mu) \in Q^\text{ad}_{\Delta t}(0,T; V_\text{RB})$ associated to $u_\text{RB}(\mu)$ will be obtained by solving
\begin{align}
\label{eq:dual_ROM_PDE}
\begin{split}
\frac{1}{\Delta t}(v, p_\text{RB}^k(\mu)-p_\text{RB}^{k+1}(\mu))_{L^2(\Omega)} + a(v, p_\text{RB}^k(\mu); \mu) &= 2d(u_\text{RB}^k(\mu), v) + l^k(v), \\
p_\text{RB}^{K + 1}(\mu) &= 0, 
\end{split}
\end{align}
for all $k \in \mathbb{K}$ and $v \in V_\text{RB}$. 
The RB objective functional follows naturally from these definitions.
\begin{definition}
\label{def:J_RB}
Let $\mu \in \mathcal{P}$ and $V_\text{RB}$ as defined in \eqref{eq:def_RB_space} be given. The \textit{\Gls{RB} approximate objective functional and gradient} are defined as
\begin{equation*}
\mathcal{J}_\text{RB}(\mu) := J(u_\text{RB}(\mu); \mu) \text{ and } d_{\mu_i} \mathcal{J}_\text{RB}(\mu) = 
d_{\mu_i}\mathcal{L}(u_\text{RB}(\mu), \mu, p_\text{RB}(\mu)),
\end{equation*}
where $u_\text{RB}(\mu) \in Q^\text{pr}_{\Delta t}(0,T; V_\text{RB})$ is the solution of
\eqref{eq:primal_ROM_PDE} and $p_\text{RB}(\mu) \in Q^\text{ad}_{\Delta t}(0,T; V_\text{RB})$ of \eqref{eq:dual_ROM_PDE}, associated to $u_\text{RB}(\mu)$. The model returning these approximations will be referred to as \textit{reduced basis-reduced order model (RB-ROM)}.
\end{definition}

\subsection{A Posteriori Error Estimator}
\label{sssec:a_post_err_est}

For \gls{TR} methods, \gls{RB} approximations come with online-efficient a posteriori error estimators for the approximation error $|\mathcal{J}_h(\mu) - \mathcal{J}_\text{RB}(\mu)|$ and its gradient. A model is considered \textit{certified} if such an error estimator exists. We summarize key aspects of \gls{RB} error estimation here; for further details, see \cite{qian2017certified, grepl2005posteriori, haasdonk_min_theta}. We begin by defining the residual operators for the reduced linear problems \eqref{eq:primal_ROM_PDE} and \eqref{eq:dual_ROM_PDE}.

\begin{definition}[Residual Operators]
\label{def:residual_ops}
For $k \in \mathbb{K}$, define the residual operator as
\begin{gather*}
r^k_\text{pr} : Q^\text{pr}_{\Delta t}(0,T; V_h) \times V_h \times \mathcal{P} \to \mathbb{R},\\
(u, v ; \mu) \mapsto r^k_\text{pr}(u, v ; \mu) := b(t^k)f(v; \mu) - a(u^k, v; \mu) - \frac{1}{\Delta t} (u^k - u^{k - 1}, v )_{L^2(\Omega)}
\end{gather*}
for the primal case and
\begin{gather*}
r^k_\text{ad} : Q^\text{pr}_{\Delta t}(0,T; V_h) \times Q^\text{ad}_{\Delta t}(0,T; V_h) \times V_h \times \mathcal{P} \to \mathbb{R},\\
(u, p, v ; \mu) \mapsto r^k_\text{ad}(u, p, v ; \mu) := 2 d(u^k, v) + l^k(v) - a(v, p^k; \mu)- \frac{1}{\Delta t} (v, p^k - p^{k + 1})_{L^2(\Omega)}
\end{gather*}
for the adjoint case. Define
\begin{gather*}
r_\text{pr}(u, v ; \mu) := \left(r^k_\text{pr}(u, v ; \mu) \right)_{k \in \mathbb{K}} \in \mathbb{R}^K,\\
r_\text{ad}(u, p, v ; \mu) := \left(r^k_\text{ad}(u, p, v ; \mu) \right)_{k \in \mathbb{K}} \in \mathbb{R}^K.
\end{gather*}
\end{definition}
To facilitate further discussion, we define the operators
\begin{gather*}
\mathcal{S} : V_h^K \to \mathbb{R}, \quad u \mapsto \left(\Delta t \sum_{k = 1}^K \|u^k\|_{V_h}^2 \right)^{1/2},\\
\mathcal{T} : (V_h')^K \to \mathbb{R}, \quad f(\cdot) \mapsto \left(\Delta t \sum_{k = 1}^K \|f^k(\cdot)\|_{V'_h}^2 \right)^{1/2}.
\end{gather*}
Before defining a posteriori error estimators for $\mathcal{J}_\text{RB}$ and $\nabla_\mu \mathcal{J}_\text{RB}$, estimators for the reduced solutions must be established.
\begin{lemma}[Error Bounds]
\label{lem:error_bound_op}
Let the trajectory $u \in Q^\text{pr}_{\Delta t}(0,T; V_\text{RB})$ be given and let ${p(u; \mu) \in Q^\text{ad}_{\Delta t}(0,T; V_\text{RB})}$ solve \eqref{eq:dual_ROM_PDE} for $\mu \in \mathcal{P}$ and the right-hand side defined by $u$. Furthermore, let $u_h(\mu) \in Q^\text{pr}_{\Delta t}(0,T; V_h)$ and $p_h(\mu) \in Q^\text{ad}_{\Delta t}(0,T; V_h)$ be the high-fidelity solutions obtained by the \gls{FOM}. We then have for $e_\text{pr}(\mu) := u_h(\mu) - u(\mu)$ and $e_\text{ad}(\mu) := p_h(\mu) - p(u; \mu)$
\begin{gather}
\label{eq:primal_error_est}
\mathcal{S}(e_\text{pr}(\mu)) \leq \Delta^\text{pr}(u, \mu) := \alpha^{-1}_\text{LB}(\mu) \mathcal{T}(r_\text{pr}(u,\cdot\,; \mu))
\end{gather}
and
\begin{align}
\begin{split}
\label{eq:dual_error_est}
\mathcal{S}(e_\text{ad}(\mu)) &\leq \Delta^\text{ad}(u, \mu) 
 := \alpha^{-1}_\text{LB}(\mu) \left(8\gamma_d^2 (\Delta^\text{pr}(u, \mu))^2+ 2 \mathcal{T}^2(r_\text{ad}(u, p(u; \mu), \cdot\,; \mu)) \right)^{\frac{1}{2}}.\qquad
\end{split}
\end{align}
\end{lemma}

\begin{proof}
See Lemma 8 in \cite{qian2017certified} and apply the Cauchy-Schwarz inequality.
\end{proof}
Given $\Delta^\text{pr}(u, \mu)$ and $\Delta^\text{ad}(u, \mu)$ for $u \in Q^\text{pr}_{\Delta t}(0,T; V_\text{RB})$, the errors of the associated \Gls{RB} objective functional and its gradient can be derived as follows.
\begin{theorem}[A posteriori objective error estimates]
\label{thm:obj_error_op}
For $\mu \in \mathcal{P}$, let the high-fidelity cost functional $\mathcal{J}_h$ and the \gls{RB} approximation $\mathcal{J}_\text{RB}$ be defined as in \eqref{eq:def_J_h} and Definition \ref{def:J_RB}. Then, the following bounds hold:
\begin{gather*}
| \mathcal{J}_h(\mu) - \mathcal{J}_\text{RB}(\mu) | \leq \Delta^{\mathcal{J}_\text{RB}}(\mu),
\end{gather*}
and
\begin{gather*}
|d_{\mu_i}\mathcal{J}_h(\mu) - d_{\mu_i}\mathcal{J}_\text{RB}(\mu)| \leq \Delta^{d_{\mu_i}\mathcal{J}_\text{RB}}(\mu),
\end{gather*}
for $i \in \{1, \dots, P\}$, where the right-hand sides are given by
\begin{equation*}
\Delta^{\mathcal{J}_\text{RB}}(\mu) := \mathcal{T}(r_\text{ad}(u_\text{RB}(\mu),p_\text{RB}(\mu),\cdot\,; \mu))\Delta^\text{pr}_\text{RB}(\mu) + \gamma_d\Delta^\text{pr}_\text{RB}(\mu)^2
\end{equation*}
and
\begin{align*}
\Delta^{d_{\mu_i}\mathcal{J}_\text{RB}}(\mu) := & \mathcal{T}(f_{\mu_i}(\cdot\,; \mu)) \Delta^\text{ad}_\text{RB}(\mu) + \gamma_{a_{\mu_i}}^{\text{UB}}(\mu)\Delta^\text{pr}_\text{RB}(\mu) \Delta^\text{ad}_\text{RB}(\mu) \\
& + \gamma_{a_{\mu_i}}^{\text{UB}}(\mu) \Delta^\text{pr}_\text{RB}(\mu)\mathcal{S}(p_\text{RB}(\mu)) + \gamma_{a_{\mu_i}}^{\text{UB}}(\mu) \Delta^\text{ad}_\text{RB}(\mu) \mathcal{S}(u_\text{RB}(\mu)),
\end{align*}
where we define $\Delta^\text{pr}_\text{RB}(\mu) := \Delta^\text{pr}(u_\text{RB}(\mu), \mu)$, $\Delta^\text{ad}_\text{RB}(\mu) := \Delta^\text{ad}(u_\text{RB}(\mu), \mu)$, as well as ${\mathcal{T}(f_{\mu_i}(\cdot\,; \mu)) := \mathcal{T}((f_{\mu_i}(\cdot\,; \mu))_{k \in \mathbb{K}})}$.
\end{theorem}

\begin{proof}
The proof follows analogously from \cite[Theorems 9 and 10]{qian2017certified}. Unlike \cite{qian2017certified}, we omit the last summand in the definition of $\Delta^{\mathcal{J}_\text{RB}}(\mu)$:
\begin{equation}
\label{eq:last_term_J_RB_error_est}
\Delta t \left|\sum_{k = 1}^K r^k_\text{pr}(u^k_\text{RB}(\mu),p^k_\text{RB}(\mu);\mu) \right|.
\end{equation}
This term is omitted since we restrict ourselves to a \gls{RB} approach where the reduced primal and adjoint problem share identical vector spaces, ensuring that \eqref{eq:last_term_J_RB_error_est} is always zero.
\end{proof}

\begin{remark}
\label{rem:accurate_RB}
If the high-fidelity solutions $u_h(\mu)$ and $p_h(\mu)$ belong to $V_{\text{RB}}$, then the \gls{RB} approximations solving \eqref{eq:primal_ROM_PDE} and \eqref{eq:dual_ROM_PDE} coincide with them. Consequently, we obtain $\Delta^{\mathcal{J}}(\mu) = 0$ and $\Delta^{\nabla_{\mu}J}(\mu) := \left(\Delta^{d_{\mu_1}\mathcal{J}_\text{RB}}(\mu), \dots, \Delta^{d_{\mu_P}\mathcal{J}_\text{RB}}(\mu) \right) = 0$.
\end{remark}

\subsection{\gls{RB} Space Construction}
We now discuss constructing a suitable basis for $V_{\text{RB}}$. The objective here is to maintain the smallest possible dimension while accurately approximating high-fidelity solutions for all parameter values. The basis of $V_{\text{RB}}$ is expanded only when the approximation error for a queried parameter $\tilde{\mu}$, determined by the aforementioned estimators, exceeds a prescribed tolerance. In such cases, an adaptation procedure updates $V_{\text{RB}}$ and the \gls{RB}-\gls{ROM}. The common approach utilizes high-fidelity trajectories for selected parameters (\textit{snapshots}) to construct the \gls{RB} space, ensuring accurate \gls{RB} solutions (see Remark \ref{rem:accurate_RB}).

Given a potentially uninitialized (i.e., the basis may be empty) reduced space $V_{\text{RB}}$, the adaptation process at $\tilde{\mu} \in \mathcal{P}$ is as follows.

\begin{procedure__}
\label{proc:RB_update}
\hspace{0.4cm}
\begin{enumerate}
\item Solve the \gls{FOM} at $\tilde{\mu}$ to obtain the primal and adjoint trajectories $u_h(\tilde{\mu})$ and $p_h(\tilde{\mu})$. Define
\begin{equation*}
\mathcal{S}^\text{pr} := \left[u_h^1(\tilde{\mu}) - \Pi_{V_\text{RB}} u_h^1(\tilde{\mu})\,\big|\,\dots\,\big|\, u_h^K(\tilde{\mu}) - \Pi_{V_\text{RB}} u_h^K(\tilde{\mu}) \right] \in \mathbb{R}^{N_h \times K},
\end{equation*}
where $\Pi_{V_\text{RB}}$ is the $\langle \cdot, \cdot \rangle_{V_h}$-projection into $V_\text{RB}$. The adjoint trajectory defines $\mathcal{S}^\text{ad}$ analogously.

\item Apply the \gls{HaPOD} \cite{himpe2018hierarchical} separately to $\mathcal{S}^\text{pr}$ and $\mathcal{S}^\text{ad}$, generating sets of $\langle \cdot, \cdot \rangle_{V_h}$-orthogonal vectors $\Phi^\text{pr} := \{ \varphi^\text{pr}_1, \dots, \varphi^\text{pr}_r \}$ and $\Phi^\text{ad} := \{ \varphi^\text{ad}_1, \dots, \varphi^\text{ad}_{\tilde{r}} \}$, satisfying
\begin{equation*}
\sum_{s \in \mathcal{S}^\text{pr}} \|s - \Pi_{\text{span} \{ \varphi^\text{pr}_1, \dots, \varphi^\text{pr}_r \}}s \|^2_{V_h} < \epsilon_\text{pod}, \quad
\sum_{s \in \mathcal{S}^\text{ad}} \|s - \Pi_{\text{span} \{ \varphi^\text{ad}_1, \dots, \varphi^\text{ad}_{\tilde{r}} \}}s \|^2_{V_h} < \epsilon_\text{pod}.
\end{equation*}

\item Extend $V_\text{RB}$ with $\Phi^\text{pr}$ and $\Phi^\text{ad}$, followed by reorthogonalization, defining an updated $V_h$-orthogonal basis, with dimension $\tilde{N}_\text{RB}$.

\item Precompute and store all parameter-independent properties. This includes $(\phi_i, \phi_j)_{L^2(\Omega)}$, $a^{q}(\phi_i, \phi_j)$ and $f^{\tilde{q}}(\phi_i)$, with $q \in Q_a$ and $\tilde{q} \in Q_f$ for all ${i, j \in \{1, \dots, \tilde{N}_\text{RB} \}}$ and for both projected \glspl{PDE} as well as the Riesz-representatives used in the error estimator \eqref{eq:primal_error_est} and \eqref{eq:dual_error_est}. The precomputation 
is carried out as described in \cite{buhr2014numerically}, by projecting the 
operator in the residual to an approximate image basis. However, unlike 
\cite{buhr2014numerically}, we omit orthogonalization for this basis to reduce 
the computational burden of the offline calculations.
\end{enumerate}
\end{procedure__}

The precomputations in Step 4 significantly reduce the online computational cost for parameter inference. This separation of parameter-independent and parameter-dependent computations is known as \textit{offline/online decomposition}. Notably, for any tolerance $\epsilon_\text{POD}$, the sets $\Phi^\text{pr}$ and $\Phi^\text{ad}$ in Step 2 can be chosen to satisfy the required accuracy constraints, ensuring that high-fidelity solutions and their \gls{RB} approximations match at $\tilde{\mu}$ up to machine precision when $\epsilon_\text{POD}$ is appropriately selected.

\section{Machine learning surrogates}
\label{sec:ML_surrogates}

Although \Gls{RB} methods are a powerful technique for reducing computational complexity, solving the \Gls{PDE}-constraint and adjoint problem remains a significant bottleneck, particularly for problems on long time scales. One strategy to enhance the \Gls{RB} approach, as discussed in, e.g., \cite{fresca2022pod, hesthaven2018non, kutyniok2020theoretical}, involves replacing the \Gls{RB}-\Gls{ROM} with machine learning surrogate models which approximate the \Gls{RB} solution operator using a (supervised)  \Gls{ML}  algorithm. \Gls{ML} is a vast and rapidly evolving field, especially over the last decade. Here, we provide a concise and general introduction and refer the reader to the extensive literature for further details, e.g., \cite{GoodBengCour16, lecun2015deep, berner2021modern, kutyniok2022mathematics}.

Let $\mathcal{X}$ and $\mathcal{Y}$ be measurable spaces, and let $\mathcal{M}(\mathcal{X}, \mathcal{Y})$ denote the set of measurable functions mapping from $\mathcal{X}$ to $\mathcal{Y}$. In general terms, supervised machine learning aims to \textit{learn} a function $T : \mathcal{X} \to \mathcal{Y}$ that approximates an (often unknown) target function $\tilde{T} : \mathcal{X} \to \mathcal{Y}$. The learning process relies on a finite training dataset ${\mathcal{M}_\text{train} \subseteq (\mathcal{X} \times \mathcal{Y})^{N_\text{train}}}$, consisting of $N$ input-output pairs
$
\mathcal{M}_\text{train} = \{(x_1, \tilde{T}(x_1)), \ldots, (x_N, \tilde{T}(x_{N_\text{train}}))\}.
$
Based on this data, a machine learning algorithm selects a function from a hypothesis set ${\mathcal{F} \subset \mathcal{M}(\mathcal{X}, \mathcal{Y})}$ of admissible functions. The selected function should best fit the training data while capturing the global behavior of $\tilde{T}$. It is crucial to note that this does not necessarily mean minimizing the error on the training data; rather, the objective is to generalize well to unseen data.

Our goal is to learn the parameter-dependent coefficients of \Gls{RB} solutions. Let $V_\text{RB}$ be as defined above, with $N_\text{RB} = \dim V_\text{RB}$. Define the primal \Gls{RB} solution operator:
\begin{align}
\label{eq:RB_sol_op}
A_\text{RB}^\text{pr}: \mathcal{P} & \rightarrow Q^\text{pr}_{\Delta t}(0,T;V_\text{RB}^\text{pr}); \qquad
\mu  \mapsto u_\text{RB}(\mu) = \left(\sum_{n = 1}^{N_\text{RB}} \underline{u^{k}_\text{RB}(\mu)}_n \psi_{n} \right)_{k \in \{0, \dots, K\}}.
\end{align}
For any $\mu \in \mathcal{P}$, this operator returns the solution to \eqref{eq:primal_ROM_PDE}. Our objective is to learn a function
\begin{equation*}
T: \mathcal{P} \rightarrow \mathbb{R}^{(K+1)N_\text{RB}}
\end{equation*}
that directly maps parameters to approximative \Gls{DoF} vectors for all times and basis vectors, i.e.
\begin{equation}
\label{eq:trained_func}
\underline{u^k_\text{ML}(\mu)}_n := T(\mu)_{kN_\text{RB} + n}, \quad 0 \leq k \leq K, \quad 1 \leq n \leq N_\text{RB}.
\end{equation}

This methodology corresponds to the time-vectorized layout described in \cite{haasdonk2022new}. Similar to the \Gls{RB} approach, we define an \Gls{ML} approximative objective functional. However, a key difference is in computing the gradient. While for the \Gls{RB}-\Gls{ROM} equation \eqref{eq:nabla_J_tilde_op} applies to the reduced solutions of \eqref{eq:primal_ROM_PDE} and \eqref{eq:dual_ROM_PDE}, this is generally not the case for \Gls{ML} approximations.
The reason is that $u_\text{ML}(\mu)$ and $p_\text{ML}(\mu)$ do not exactly solve the reduced equations, which is a prerequisite for using the adjoint methodology \cite[Section 1.6.1]{hinze2009optimization}. We therefore demand $C^1$-regularity for $\mu \mapsto \underline{u^{k}_\text{ML}(\mu)}_n$ and choose an \Gls{ML} model appropriately, computing gradients directly via the chain rule.

\begin{definition}
\label{def:ML_surrogate}
Let $T: \mathcal{P} \rightarrow \mathbb{R}^{(K+1)N_\text{RB}}$ be a learned, continuously differentiable function, and define $\underline{u^k_\text{ML}(\mu)}_n$ as in \eqref{eq:trained_func}. Assume that
\begin{equation*}
\underline{u^0_\text{ML}}_n \equiv 0 \text{ and } d_{\mu_i} \underline{u^0_\text{ML}}_n \equiv 0 \text{ for } 1 \leq n \leq N_\text{RB}.
\end{equation*}
Then, the \textit{ML approximative primal solution} is given by
\begin{equation}
\label{eq:def_u_ML}
u_\text{ML}(\mu) := \left(\sum_{n = 1}^{N_\text{RB}} \underline{u^{k}_\text{ML}(\mu)}_n \psi_{n} \right)_{k \in \{0, \dots, K\}} \in Q^\text{pr}_{\Delta t} (0, T; V_\text{RB}).
\end{equation}
The respective partial derivatives are defined as
\begin{equation}
\label{eq:def_du_ML}
d_{\mu_i} u_\text{ML}(\mu) = \left( \sum_{n = 1}^{N_\text{RB}}
d_{\mu_i} \underline{u^{k}_\text{ML}(\mu)}_n \psi_{n}\right)_{k \in \{0, \dots, K\}} \in Q^\text{pr}_{\Delta t} (0, T; V_\text{RB}).
\end{equation}
The \Gls{ML} approximative objective functional, analogous to Definition \ref{def:J_RB}, is given by
\begin{equation*}
\mathcal{J}_\text{ML}(\mu) := J(u_\text{ML}(\mu);\mu).
\end{equation*}
Using the chain rule, the partial derivatives of $\mathcal{J}_\text{ML}(\mu)$ are
\begin{align*}
d_{\mu_i} \mathcal{J}_\text{ML}(\mu) = \Delta t \sum_{k = 1}^K \left[2d(u^k_\text{ML}(\mu), d_{\mu_i} u^k_\text{ML}(\mu)) + l^k(d_{\mu_i} u^k_\text{ML}(\mu))\right] + \lambda d_{\mu_i}\mathcal{R}(\mu),
\end{align*}
for $i \in \{1, \dots, P\}$, defining the \Gls{ML}-\Gls{ROM}.
\end{definition}

\subsection{Error Estimates for ML Surrogates}
\label{ssec:error_est_ML}

Similarly to the \gls{RB} case, we aim to certify the \gls{ML}-\gls{ROM} by defining a posteriori error estimators that provide upper bounds for the model error between $\mathcal{J}_\text{ML}$, $\nabla_\mu \mathcal{J}_\text{ML}$, and their high-fidelity counterparts. Following Lemma \ref{lem:error_bound_op}, an a posteriori error bound for $u_\text{ML}(\mu)$, as defined in \eqref{eq:def_u_ML}, can be established such that for all $\mu \in \mathcal{P}$, the following holds:
\begin{gather}
\label{eq:u_ML_error_bound}
\mathcal{S}(e(\mu)) \leq \Delta^\text{pr}(u_\text{ML}(\mu), \mu) = \alpha^{-1}_\text{LB}(\mu) \mathcal{T}(r_\text{pr}(u_\text{ML}(\mu),\cdot\,; \mu)),
\end{gather}
where $e(\mu) := u_h(\mu) - u_\text{ML}(\mu)$, and $u_h(\mu) \in Q^\text{pr}_{\Delta t}(0,T; V_h)$ is the primal solution obtained via the \gls{FOM}.

Similarly, an a posteriori error bound can be formulated for $d_{\mu_i} u_\text{ML}(\mu)$ by adapting the proof of Lemma \ref{lem:error_bound_op} in \cite{qian2017certified} and Proposition 4.1 in \cite{grepl2005posteriori}, see Appendix \ref{sec:appendix_proofs}.

\begin{lemma}
\label{lem:error_bound_u_ML}
Let $u_h(\mu) \in Q^\text{pr}_{\Delta t}(0,T; V_h)$ be the high-fidelity solution to \eqref{eq:primal_analytic_PDE}, and let $u_\text{ML}(\mu) \in Q^\text{pr}_{\Delta t}(0,T; V_h)$ be defined as in Definition \ref{def:ML_surrogate}, with derivative $d_{\mu_i}u_\text{ML}(\mu)$. Define $d_{\mu_i}u_h(\mu)$ as the (well-defined) Fr\'{e}chet derivatives of $u_h(\mu)$. For $k \in \mathbb{K}$, set
\begin{align}
\label{eq:residuum_dmu_ML}
\begin{split}
\mathcal{R}^k_\text{ML}(v; \mu) := & d_{\mu_i}r^k_\text{pr}(u^k_\text{ML}, v;\mu) - a(d_{\mu_i}u^k_\text{ML}(\mu), v; \mu) \\
& - \frac{1}{\Delta t }(d_{\mu_i}u_\text{ML}^k(\mu)- d_{\mu_i}u_\text{ML}^{k-1}(\mu), v)_{L^2(\Omega)},
\end{split}
\end{align}
for all $v \in V_h$. Then, for all $i \in \{1, \dots, P\}$, we have
\begin{equation*}
\mathcal{S}(d_{\mu_i}e(\mu)) \leq \Delta^{d_{\mu_i}u}_\text{ML}(\mu) := \alpha^{-1}_\text{LB}(\mu) \left(\mathcal{T}(\mathcal{R}_\text{ML}(\cdot; \mu)) + \gamma_{a_{\mu_i}}(\mu) \Delta^\text{pr}_\text{ML}(\mu)\right),
\end{equation*}
where we define ${d_{\mu_i}e(\mu):= d_{\mu_i}u_h(\mu) - d_{\mu_i}u_\text{ML}(\mu)}$, ${
\Delta^\text{pr}_\text{ML}(\mu) := \Delta^\text{pr}(u_\text{ML}(\mu), \mu)}$ and 
${\mathcal{R}_\text{ML}(\cdot\,; \mu) := \left(\mathcal{R}^k_\text{ML}(\cdot\,; \mu) \right)_{k \in \mathbb{K}} \in \left(V'_h\right)^K}$.
\end{lemma}
Using this lemma, an a posteriori error bound  can be formulated for $\mathcal{J}_\text{ML}(\mu)$ and ${\nabla_\mu\mathcal{J}_\text{ML}(\mu) := (d_{\mu_1} \mathcal{J}_\text{ML}(\mu), \dots, d_{\mu_P} \mathcal{J}_\text{ML}(\mu))}$, similar to Theorem \ref{thm:obj_error_op}. The proof is presented in Appendix \ref{sec:appendix_proofs}.
\begin{theorem}
\label{thm:error_est_ML}
Let $\mu \in \mathcal{P}$, and let $\mathcal{J}_h(\mu)$, $d_{\mu_i} \mathcal{J}_h(\mu)$ be defined as in Theorem \ref{thm:obj_error_op}. Then, we have
\begin{gather}
\label{eq:J_ML}
| \mathcal{J}_h(\mu) - \mathcal{J}_\text{ML}(\mu) | \leq \Delta^{\mathcal{J}_\text{ML}}(\mu),
\end{gather}
and
\begin{gather}
\label{eq:nabla_J_ML}
|d_{\mu_i} \mathcal{J}_h(\mu) - d_{\mu_i} \mathcal{J}_\text{ML}(\mu)| \leq \Delta^{d_{\mu_i} \mathcal{J}_\text{ML}}(\mu),
\end{gather}
for all $i \in \{1, \dots, P\}$. The error bounds on the right-hand side are given by
\begin{equation*}
\Delta^{\mathcal{J}_\text{ML}}(\mu) := \left[ 2 \mathcal{S}(u_\text{ML}(\mu)) + \mathcal{S}(g_\text{ref}(\mu)) \right] \gamma_d \Delta^\text{pr}_\text{ML}(\mu) + \gamma_d \Delta^\text{pr}_\text{ML}(\mu)^2,
\end{equation*}
and
\begin{align*}
\Delta^{d_{\mu_i} \mathcal{J}_\text{ML}}(\mu) := & \left[2\Delta^\text{pr}_\text{ML}(\mu) + 2\mathcal{S}(u_\text{ML}(\mu)) + \mathcal{S}(g_\text{ref})\right] \gamma_d \Delta^{d_{\mu_i}u}_\text{ML}(\mu) \\
& + 2 \mathcal{S}(d_{\mu_i}u_\text{ML}(\mu))  \gamma_d\Delta^\text{pr}_\text{ML}(\mu).
\end{align*}
\end{theorem}

\subsection{General Considerations}
\label{ssec:ML_general_considerations}

Replacing the computation of the \gls{RB} solutions with their \gls{ML} counterparts offers two main advantages. Firstly, parameter inference using the \gls{ML}-\gls{ROM} is typically significantly faster \cite{haasdonk2022new, wenzel2023application} compared to the \gls{RB}-\gls{ROM}. Secondly, this approach allows for parallel computation since the solution $\underline{u^{k}_\text{ML}(\mu)}_n$ at a given time step does not explicitly depend on previous time steps, enabling simultaneous calculation of all time steps. However, to efficiently use an \gls{ML}-based model for an optimization task, several key aspects regarding the training process and the specific considerations of incorporating an \gls{ML} model into a gradient descent type optimization algorithm must be addressed. In this section, we outline the requirements an \gls{ML} algorithm must satisfy to be considered viable in this context. Due to the lack of a sufficiently developed analytic theory (to the best of our knowledge), our argumentation will be primarily heuristic.

First, we must consider the domain of \gls{ML} models and its impact on training. As previously mentioned, the surrogate function $T$ is selected based on the available training data. A critical consequence of this is that the approximation quality of $T$ generally deteriorates significantly in regions of $\mathcal{P}$ with low training data density. When applying this to an \gls{RB} optimization algorithm, we encounter the challenge that obtaining sufficient training data for a globally accurate \gls{ML} model requires evaluating the \gls{RB} models for a sufficiently large set of training parameters $\mathcal{P}_\text{train}$, distributed \textit{uniformly} in $\mathcal{P}$. However, this approach would lead to an impractically large $\mathcal{P}_\text{train}$ and, consequently, an excessive number of \gls{RB} model evaluations, rendering offline computations infeasibly costly. This issue becomes even more pronounced as the parameter space dimension $P$ increases, a phenomenon commonly referred to as the \textit{Curse of Dimensionality}. Furthermore, we must ensure that the \gls{RB} models themselves remain globally accurate with respect to the high-fidelity models. This further increases offline computational costs, as the reduced basis space $V_\text{RB}$ must be sufficiently enriched to achieve the desired accuracy globally.

To address this challenge, we propose an algorithmic framework incorporating localized \gls{ML} models, following the hierarchical modeling approach outlined in \cite{haasdonk2022new, kleikamp2024application, kleikamp2024adaptivemodelhierarchiesmultiquery}, adapted specifically for optimization tasks. In this framework, a hierarchy of models is established in decreasing order of accuracy and (generally) increasing inference speed. The key idea is to always attempt the fastest model first, only resorting to slower, more accurate models when necessary. Results from more precise models are then leveraged to iteratively enhance the less accurate models. Specifically, our approach first evaluates the \gls{ML}-\gls{ROM}; if its output lacks sufficient accuracy, we proceed to the \gls{RB}-\gls{ROM}. Should this also prove inadequate, the \acrlong{FOM} is employed. The \gls{FOM} solutions are subsequently used to refine $V_\text{RB}$, while the \gls{RB} solutions serve as training data for the \gls{ML}-\gls{ROM}. Unlike the approach in \cite{haasdonk2022new, kleikamp2024adaptivemodelhierarchiesmultiquery}, where re-adaptation is guided by an a posteriori error estimate, our re-training process is triggered by a decay condition of the objective functional, as detailed in Section \ref{sec:TR}. A crucial distinction of our approach is that we do not employ a distinct offline and online phase but instead alternate between model improvement and evaluation. This eliminates the need for an expensive offline phase, such as generating training data for a globally accurate \gls{ML} model. Instead, we start with coarse models and refine them dynamically using intermediate results, adapting them to the domain of interest, which often reduces computational complexity compared to a priori constructed models.

Following this dynamic adaptation strategy, \gls{ML} models are trained only with \gls{RB} solutions collected during the optimization process. Consequently, \gls{ML} models are expected to be sufficiently accurate only near the optimization trajectory. This necessitates careful selection of the learning method and imposes additional requirements to ensure the surrogate model's quality. First, the learning procedure must perform well with relatively few localized training data (I) and must converge rapidly to the learned function (II). The latter is particularly important because, due to localization, \gls{ML} models generally require multiple retraining steps, and excessive training overhead could negate the computational speed advantage. Additionally, we impose an interpolation property requirement (III), ensuring that the learned function satisfies
\begin{equation}
\label{eq:ml_interpol_error}
\max_{i \in \{1, \dots, N_\text{train}\}} \|y_i - T_\eta(x_i) \|_{\mathcal{Y}} \leq \eta,
\end{equation}
where $\|\cdot\|_{\mathcal{Y}}$ is an appropriate norm on $\mathcal{Y}$. This property, combined with the continuity provided by the $C^1$-regularity condition (see Definition \ref{def:ML_surrogate}), ensures approximation accuracy in a small region surrounding the training inputs $x_1, \dots, x_N$. However, it should be noted that, in general, the size of the region where the \gls{ML} surrogates remain sufficiently accurate cannot be estimated reliably.

\section{Trust region multi-fidelity optimization}
\label{sec:TR}
In this section, we present a \acrlong{TR} algorithm that integrates both \Gls{RB}-\glspl{ROM} and \Gls{ML}-\glspl{ROM}. As previously discussed, the \Gls{RB}-\Gls{ROM} serves as the foundation, providing both training data and certification. Following the approach outlined in \cite{qian2017certified}, the global optimization problem \ref{prob:t_dis_opt_prob} is reformulated as a sequence of iteratively solved sub-problems:

\begin{equation} 
\label{eq:TR_subproblem} 
\mu^{(i+1)} := \argmin_{\mu \in \mathcal{P}} \mathcal{J}^{(i)}(\mu) \quad \text{s.t.} \quad \mu \in T^{(i)}, 
\end{equation}
which are solved until a global convergence criterion is met. Here, $\mathcal{J}^{(i)}(\cdot)$ represents a local approximation of $\mathcal{J}(\cdot)$ for each iteration index $i \in \mathbb{I} := {0, \dots, I-1}$, for ${I \in \mathbb{N}_0 \cup \{\infty\}}$. Let $M^{(i)}$ denote models that provide access to local approximations $\mathcal{J}^{(i)}(\cdot)$ and their gradients $\nabla_\mu \mathcal{J}^{(i)}(\cdot)$. These models are updated iteratively to adapt to the current trust region. To ensure global convergence, we assume that for each sub-problem \eqref{eq:TR_subproblem}, a sequence $\lbrace \mu^{(i,l)} \rbrace_{l = 0}^{L^{(i)}}$ can be found beginning with $\mu^{(i,0)} := \mu^{(i)}$, such that

\begin{equation} \label{eq:inner_decay_condition} \mathcal{J}^{(i)}(\mu^{(i, l)}) \geq \mathcal{J}^{(i)}(\mu^{(i, l+1)}) \quad \text{for all} \quad l \in {0, \dots, L^{(i)} - 1}. 
\end{equation}
This ensures a monotonic decrease in the objective function, driving the optimization process towards convergence.
The final iteration point is used as the starting point for the next sub-problem, i.e., $\mu^{(i+1)} := \mu^{(i,L^{(i)})}$. 

 However, condition \eqref{eq:inner_decay_condition} alone does not guarantee a global decrease in the objective function. We therefore require that
\begin{equation}
\label{eq:outer_decay_condition}
\mathcal{J}^{(i)}(\mu^{(i)}) \geq \mathcal{J}^{(i+1)}(\mu^{(i+1)}) \quad \text{for all } i \in \mathbb{I}.
\end{equation}
To ensure \eqref{eq:outer_decay_condition}, we impose the \textit{error-aware sufficient decrease condition}, given by:
\begin{equation}
\label{eq:EASDC}
\mathcal{J}^{(i)}(\mu_\text{AGC}^{(i)}) \geq \mathcal{J}^{(i+1)}(\mu^{(i+1)}),
\end{equation}
in addition to \eqref{eq:inner_decay_condition}, where $\mu_\text{AGC}^{(i)} := \mu^{(i,1)}$ represents the \Gls{AGC}. If \eqref{eq:EASDC} holds, then \eqref{eq:outer_decay_condition} is ensured, and $\mu^{(i+1)}$ is accepted. Otherwise, the guess $\mu^{(i+1)}$ is rejected, and the $i$-th sub-problem is resolved with a shrunken trust region. To verify \eqref{eq:EASDC} without the  expensive construction of $M^{(i+1)}$, we assume that for all $ i \in \mathbb{I}$, error bounds $\Delta^{\mathcal{J},(i)}(\mu^{(i+1)})$ and $\Delta^{\nabla \mathcal{J},(i)}(\mu^{(i+1)})$ are available, satisfying

\begin{equation*}
\left|\mathcal{J}^{(i)}(\mu) - \mathcal{J}(\mu)\right| \leq \Delta^{\mathcal{J},{(i)}}(\mu) \text{  and  } \|\nabla_\mu \mathcal{J}^{(i)}(\mu) - \nabla_\mu \mathcal{J}(\mu)\|_{\mathbb{R}^P} \leq \Delta^{\nabla \mathcal{J},{(i)}}(\mu).
\end{equation*}

With these error estimates, the error-aware sufficient decrease condition can be refined by checking the following sufficient condition:
\begin{equation}
\label{eq:suff_cond}
\mathcal{J}^{(i)}(\mu^{(i+1)}) + \Delta^{\mathcal{J}, (i)}(\mu^{(i+1)}) < \mathcal{J}^{(i)}(\mu_\text{AGC}^{(i)}),
\end{equation}
and the necessary condition:
\begin{equation}
\label{eq:nec_cond}
\mathcal{J}^{(i)}(\mu^{(i+1)}) - \Delta^{\mathcal{J},{(i)}}(\mu^{(i+1)})\leq \mathcal{J}^{(i)}(\mu_\text{AGC}^{(i)}),
\end{equation}
before going to \eqref{eq:EASDC}, cf. \cite{yue2013accelerating}. The optimization process terminates when \(\mu^{(i+1)}\) satisfies a \textit{global first-order termination criterion}:
\begin{equation}
\label{eq:glob_term_crit}
\|\mu^{(i+1)} - P_\mathcal{P}(\mu^{(i+1)} - \nabla_\mu \mathcal{J}(\mu^{(i+1)}))\|_{\mathbb{R}^P} \leq \tau,
\end{equation}
where $\tau > 0$, cf. \cite{Keil_2021}. The operator $P : \mathbb{R}^P \rightarrow \mathcal{P}$ projects a potential overshot back into the admissible domain, i.e.,
\begin{equation*}
P(\mu)_j := 
\begin{cases}
L_j, & \text{if } (\mu)_j \leq L_j, \\ 
(\mu)_j, & \text{if } L_j \leq (\mu)_j \leq U_j, \\
U_j, & \text{if } (\mu)_j \geq U_j.
\end{cases}
\end{equation*}
All these elements together constitute the \Gls{TR} procedure we employ, as sketched in Figure \ref{fig:TR_illustration} and
detailed in Algorithm \ref{algo:TR_opt}.
\begin{figure}
\centering
\includegraphics[width=0.80\textwidth]{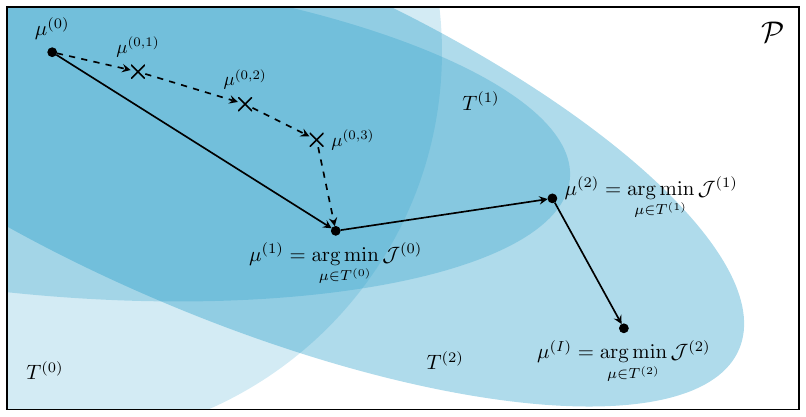}
\caption{\label{fig:TR_illustration} Schematic visualization of the abstract trust region algorithm.}
\end{figure}
\begin{algorithm2e}[h]
\DontPrintSemicolon

\caption{Trust region optimization \cite{yue2013accelerating, qian2017certified}}\label{algo:TR_opt}
Choose $\mu^{(0)} \in \mathcal{P}$, set $i := 0$ and initialize $M^{(0)}$.\;
\While{\eqref{eq:glob_term_crit} is not satisfied}{
Compute $\mu^{(i+1)}$ as solution of \eqref{eq:TR_subproblem} with a convenient termination criteria. \label{line:compute} \;
	\If{the sufficient condition \eqref{eq:suff_cond} holds}{
		Accept $\mu^{(i+1)}$.\;
		Use $\mu^{(i+1)}$ to update the model $M^{(i)}$, getting $M^{(i+1)}$ and set $T^{(i+1)} := T^{(i)}$. \label{line:updata_1}
	}\ElseIf{condition \eqref{eq:nec_cond} fails}{
		Reject $\mu^{(i+1)}$ and shrunk $T^{(i)}$, getting $T^{(i+1)} \subset T^{(i)}$. \label{line:shrink_1}
	}\Else{
		Use $\mu^{(i+1)}$ to update the model $M^{(i)}$, getting $M^{(i+1)}$. \label{line:update_2} \;
		\If{\eqref{eq:EASDC} holds}{
			\text{Accept $\mu^{(i+1)}$ and set $T^{(i+1)} := T^{(i)}$.} \;
		} \Else {
			Reject $\mu^{(i+1)}$ and shrunk $T^{(i)}$, getting $T^{(i+1)} \subset T^{(i)}$. \label{line:shrink_2} \label{line:EASDC_reject} \;
		}
	}
	$i \leftarrow i + 1$\;
}
\end{algorithm2e}

\begin{remark}
Algorithm \ref{algo:TR_opt} does not guarantee termination, as the optimization process may become trapped in an infinite loop, continuously rejecting each candidate $\mu^{(i+1)}$ due to the condition specified in line \ref{line:EASDC_reject}. However, if the algorithm does terminate, it can be shown that, under mild assumptions on $M^{(i)}$, the final iterate $\mu^{(I)}$ is, up to a given tolerance, a first-order critical point of $\mathcal{J}$, cf. \cite{yue2013accelerating}. In the following sections, we introduce suitable models $M^{(i)}$ that satisfy these conditions.
\end{remark}

\subsection{Backtracking Procedure}
We will use the \Gls{FOM} defined in Section \ref{sec:MOR} as the model, providing access to $\mathcal{J}(\mu)$ and its gradient $\nabla_\mu \mathcal{J}(\mu)$, i.e., $\mathcal{J}(\mu) := \mathcal{J}_h(\mu)$ and $\nabla_\mu \mathcal{J}(\mu) := \nabla_\mu \mathcal{J}_h(\mu)$. It remains to discuss the construction of the models $M^{(i)}$ and, based on this, the procedure for computing the sequence $\lbrace \mu^{(i,l)} \rbrace_{l = 0}^{L^{(i)}}$. These aspects are closely interdependent, particularly when employing \Gls{ML}-\Gls{ROM} approaches. As outlined in Section \ref{sec:ML_surrogates}, the \Gls{ML} surrogate requires continuous retraining, and its accuracy is therefore dependent on the parameters evaluated in previous iterations.

The algorithm we propose, is based on the projected \acrshort{BFGS} method for simple bounded domains, as introduced in \cite{kelley1999iterative, Keil_2021}. The minimizing sequence $\lbrace \mu^{(i,l)} \rbrace_{l = 0}^{L^{(i)}}$ is iteratively constructed, starting at $\mu^{(i,0)} := \mu^{(i)}$, using a backtracking procedure. For now, assume that for all $i \in \mathbb{N}_0$ and $\mu \in \mathcal{P}$, the function $\mathcal{J}^{(i)}(\mu)$ and its gradient $\nabla_\mu \mathcal{J}^{(i)}(\mu)$ are accessible. We define the update step as
\begin{equation}
\label{eq:update_backtracking}
\mu^{(i, l)}(k) := P \left(\mu^{(i, l)} + \alpha^{(i)}_0 \kappa_\text{bt}^{k} d^{(i, l)}\right) \text{ for } k \in \mathbb{N}_0,
\end{equation}
where $d^{(i, l)}$ denotes a descent direction at $\mu^{(i, l)}$, $\alpha^{(i)}_0 > 0$ is the initial step size, and ${\kappa_\text{bt} \in (0,1)}$ is a factor controlling the backtracking speed. The next iterate $\mu^{(i, l + 1)}$ is selected as $\mu^{(i, l + 1)} :=  \mu^{(i, l)}(\tilde{k})$, where $\tilde{k} \in \mathbb{N}_0$ is chosen to ensure sufficient decrease in $\mathcal{J}^{(i)}$. This decrease is enforced by verifying an Armijo-type condition:

\begin{gather}
\mathcal{J}^{(i)}(\mu^{(i, l)}) - \mathcal{J}^{(i)}(\mu^{(i, l)}(\tilde{k})) \geq \frac{\alpha_\text{arm}}{\alpha^{(i)}_0 \kappa_\text{bt}^{\tilde{k}}} \|\mu^{(i, l)} - \mu^{(i, l)}(\tilde{k})\|_{\mathbb{R}^P}^2, \label{eq:armijo}
\end{gather}
where $\alpha_\text{arm} = 10^{-6}$. Additionally, it must be ensured that $\mu^{(i, l)}(k)$ remains inside the trust region $T^{(i)}$ and that a minimum step size is satisfied to prevent unreasonably slow convergence. Hence, $\tilde{k}$ is chosen as the smallest natural number (starting from zero) satisfying \eqref{eq:armijo},
\begin{gather}
\|\mu^{(i, l)} - \mu^{(i, l)}(\tilde{k})\|_{\mathbb{R}^P} \geq \epsilon_\text{cutoff} \text{ and } \label{eq:cutoff} \\
\mu^{(i, l)}(\tilde{k}) \in T^{(i)},  \label{eq:TR_cond}
\end{gather}
for a given $\epsilon_\text{cutoff} > 0$. The iteration terminates when $\mu^{(i+1)} = \mu^{(i, L^{(i)})}$ satisfies the \textit{sub-problem termination criterion} for some $\tau_\text{sub} \in (0 , 1)$ and $\tau_\text{sub} \leq \tau$, i.e.
\begin{equation}
\label{eq:subprob_term_crit_I}
\|\mu^{(i,l)} - P_\mathcal{P}(\mu^{(i,l)} - \nabla_\mu \mathcal{J}^{(i)}(\mu^{(i,l)}))\| \leq \tau_\text{sub}.
\end{equation}

The search directions $d^{(i, l)}$ are determined using the algorithm presented in \cite[Section 5.5.3]{kelley1999iterative}, given by
\begin{equation}
\label{eq:set_d}
\tilde{d}^{(i, l)} := 
\begin{cases}
-\tilde{\mathcal{H}}^{(i, l)} \nabla_\mu \mathcal{J}^{(i)}(\mu^{(i, l)}), & \text{if } -\nabla_\mu \mathcal{J}^{(i)}(\mu^{(i, l)})^T \tilde{\mathcal{H}}^{(i, l)} \nabla_\mu \mathcal{J}^{(i)}(\mu^{(i, l)}) < 0, \\
- \nabla_\mu \mathcal{J}^{(i)}(\mu^{(i, l)}), & \text{otherwise.}
\end{cases}
\end{equation}

If $\|\tilde{d}^{(i, l)}\| \neq 0$, we normalize $d^{(i, l)} := \tilde{d}^{(i, l)} / \|\tilde{d}^{(i, l)}\|$; otherwise, we set $d^{(i, l)} = 0$. The matrix $\tilde{\mathcal{H}}^{(i, l)}$ is an iteratively constructed approximation of the inverse Hessian $(\mathcal{H}_{\mathcal{J}}(\mu^{(i,l)}))^{-1}$, leveraging previously computed gradients $\nabla_\mu \mathcal{J}^{(i)}$, where the initialization is given by $\tilde{\mathcal{H}}^{(i, 0)} := \text{Id}_{\mathbb{R}^P}$.

It is important to emphasize that the sequence $\lbrace \mu^{(i,l)} \rbrace_{l = 0}^{L^{(i)}}$ defined by equations \eqref{eq:armijo}--\eqref{eq:subprob_term_crit_I} is not necessarily well-defined. Specifically, for all $i \in \mathbb{N}_0$, the existence of a $\tilde{k}$ satisfying \eqref{eq:armijo}--\eqref{eq:TR_cond} cannot be guaranteed. Moreover, it cannot be assumed that $\mathcal{J}^{(i)}$ possesses a first-order critical point within $T^{(i)}$, nor that any iterate $\mu^{(i,l)}$ satisfies the termination criterion \eqref{eq:subprob_term_crit_I}. Therefore, it is essential for the final algorithm to incorporate mechanisms that handle these exceptions while ensuring that the sequence adheres to the decay condition \eqref{eq:inner_decay_condition}.

\subsection{\Gls{RB}-\Gls{ML}-\Gls{ROM}-optimization}

Analogous to \cite{qian2017certified, Keil_2021}, the \Gls{RB} space remains fixed for each sub-problem \eqref{eq:TR_subproblem} and is updated in the outer loop (Algorithm \ref{algo:TR_opt}; lines \ref{line:updata_1} and \ref{line:update_2}) using Procedure \ref{proc:RB_update}. This update is applied at $\mu^{(i+1)}$ by obtaining high-fidelity trajectories for $\mu^{(i+1)}$ and selecting relevant modes via the \Gls{HaPOD}-Greedy scheme. Given a fixed reduced space, we define the $i$-th \Gls{RB} surrogate model $M^{(i)}_\text{RB}$ as follows.
\begin{definition}
\label{def:RB_surrogate}
Let $i \in \mathbb{I}$ and $V^{(i)}_\text{RB}$ be given. The \Gls{RB} surrogate model $M^{(i)}_\text{RB}$ is characterized by the following components:

\begin{enumerate}
\item[(I)] \textbf{Evaluation of the Objective:}
\begin{equation*}
\mathcal{J}^{(i)}_\text{RB}(\mu), \nabla_\mu \mathcal{J}^{(i)}_\text{RB}(\mu) \leftarrow M^{(i)}_\text{RB}\texttt{.eval\_output}[\mu]
\end{equation*}
This method computes the \Gls{RB} approximate cost functional and its gradient (as per Definition \ref{def:J_RB}) by solving \eqref{eq:primal_ROM_PDE} and \eqref{eq:dual_ROM_PDE} in $V^{(i)}_\text{RB}$.

\item[(II)] \textbf{Error Estimation:}
\begin{equation*}
\Delta^{\mathcal{J}^{(i)}_\text{RB}}(\mu), \Delta^{\nabla_{\mu}\mathcal{J}^{(i)}_\text{RB}}(\mu) \leftarrow M^{(i)}_\text{RB}\texttt{.est\_output}[\mu]
\end{equation*}
This method provides error estimates associated with $M^{(i)}_\text{RB}\texttt{.eval\_output}[\mu]$, as defined in Theorem \ref{thm:obj_error_op}.

\item[(III)] \textbf{Model Extension:}
\begin{equation*}
M^{(i+1)}_\text{RB} \leftarrow M^{(i)}_\text{RB}\texttt{.extend}[\mu]
\end{equation*}
A new surrogate model $M^{(i+1)}_\text{RB}$ is constructed by extending the \Gls{RB} space $V^{(i)}_\text{RB}$ according to Procedure \ref{proc:RB_update}. That is, $M^{(i+1)}_\text{RB}$ is analogous to $M^{(i)}_\text{RB}$ but utilizes an expanded space $V^{(i+1)}_\text{RB} \supset V^{(i)}_\text{RB}$.
\end{enumerate}

Additionally, we define the training dataset:

\begin{enumerate}
\item[(IV)] \textbf{Training Data Collection:}
\begin{equation*}
\mathcal{M}_\text{train} := ((\mu_1, \underline{u_\text{RB}(\mu_1)}), \dots, (\mu_{N_\text{train}},\underline{u_\text{RB}(\mu_{N_\text{train}})}))
\end{equation*}
This set stores the \Gls{DoF} vectors from the last $N_\text{train}$ computed solutions, where: $\underline{u_\text{RB}(\mu)} := (\underline{u_\text{RB}^k(\mu)})_{k \in \{0, \dots, K\}}$.
\end{enumerate}
\end{definition}

By initializing $V^{(0)}_\text{RB} := \langle \emptyset \rangle$, we explicitly define the sequence of \Gls{RB} surrogate models $(M^{(i)}_\text{RB})_{i \in \mathbb{I}}$. Each sub-problem is uniquely associated with a surrogate model $M^{(i)}_\text{RB}$, which is iteratively refined using high-fidelity solutions obtained from previous sub-problems. Notably, the \Gls{FOM} for these solutions is evaluated only in the outer loop (Algorithm \ref{algo:TR_opt}).

In contrast, a corresponding \Gls{ML}-based surrogate model requires frequent adaptation during solving the sub-problems, as discussed in Section \ref{sec:ML_surrogates}. We now define the $m$-th \Gls{ML} surrogate model for the $i$-th sub-problem:

\begin{definition}
Let $i \in \mathbb{I}$, $m \in \mathbb{N}$, and let $M^{(i)}_\text{RB}$ be given as defined in Definition \ref{def:RB_surrogate}, with $N_\text{RB}^{(i)} := \dim  V^{(i)}_\text{RB}$. The \Gls{ML} surrogate model $M^{(i, m)}_\text{ML}$ is defined as follows:

\begin{enumerate}
\item[(I)] \textbf{Learned Operator:}
\begin{equation*}
T^{(i, m)} : \mathcal{P} \rightarrow \mathbb{R}^{(K+1)N_\text{RB}^{(i)}}
\end{equation*}
This operator maps $\mu$ to $(\underline{u_\text{ML}^k(\mu)})_{k \in \{0, \dots, K\}}$, satisfying the conditions stated in Definition \ref{def:ML_surrogate}.

\item[(II)] \textbf{Evaluation of the Objective:}
\begin{equation*}
\mathcal{J}^{(i,m)}_\text{ML}(\mu), \nabla_\mu \mathcal{J}^{(i,m)}_\text{ML}(\mu) \leftarrow M^{(i,m)}_\text{ML}\texttt{.eval\_output}[\mu]
\end{equation*}
This method computes the \Gls{ML} approximation of the objective functional and gradient (as per Definition \ref{def:ML_surrogate}) by invoking $T^{(i, m)}$.

\item[(III)] \textbf{Model Training:}
\begin{equation*}
M^{(i, m+1)}_\text{ML} \leftarrow M^{(i, m)}_\text{ML}\texttt{.train}
\end{equation*}
This method retrains $T^{(i, m)}$ using a \Gls{ML} algorithm on the dataset $\mathcal{M}_\text{train}$ collected by $M^{(i)}_\text{RB}$, yielding an updated model $T^{(i, m+1)}$. If $\mathcal{M}_\text{train}$ remains unchanged since the last training call, this training is skipped to avoid unnecessary computational overhead.
\end{enumerate}
\end{definition}

These models are initialized by setting $T^{(i, 0)}$ as the zero operator. Although error estimators for the \Gls{ML}-\Gls{ROM} were derived in Section \ref{ssec:error_est_ML}, they are not used to define $M^{(i,m)}_\text{ML}$. The primary reason is that, despite their theoretical value, these estimators are computationally inefficient due to the high inference speed of the \Gls{ML} models. Their use would significantly increase computational costs, potentially negating the speed advantage of the \Gls{ML}-\Gls{ROM}. Consequently, the models $M^{(i, m)}_\text{ML}$ remain \emph{uncertified}, posing a challenge for direct application in a \Gls{TR} algorithm, as done in \cite{qian2017certified, Keil_2021}. In those approaches, the trust region is given as

\begin{equation}
\label{eq:orig_TR}
\tilde{T}^{(i)} := \left\{ \mu \in \mathcal{P} \,\,\Bigg| \Delta^{\mathcal{J}^{(i)}_\text{RB}}_r(\mu) \,\,  \leq \epsilon^{(i)}_L \right\},
\end{equation}
for some $\epsilon_L^{(i)} > 0$, where $\Delta^{\mathcal{J}^{(i)}_\text{RB}}_r(\mu) := | \Delta^{\mathcal{J}^{(i)}_\text{RB}}(\mu) / \mathcal{J}_\text{RB}^{(i)}(\mu)|$. This formulation is not well suited for the machine learning surrogate. To address this limitation, we extend the trust region by introducing a \textit{relaxation range} of width $\kappa^{(i)} := \alpha^{(i)}_0 l_\text{check}$, where $l_\text{check} \in \mathbb{N}$, leading to the modified trust region

\begin{equation*}
T^{(i)} := \left(\tilde{T}^{(i)} \bigcup_{\mu \in \partial \tilde{T}^{(i)}} \left\{\tilde{\mu} \in  \mathbb{R}^P \mid \|\tilde{\mu} - \mu \|_{\mathbb{R}^P} \leq \kappa^{(i)} \right\}\right) \cap \mathcal{P}.
\end{equation*}
This construction eliminates the need to verify condition \eqref{eq:TR_cond} at each backtracking step. If $\Delta^{\mathcal{J}^{(i)}_\text{RB}}_r(\mu^{(i,n)}) \leq \epsilon_L^{(i)}$ holds for some $n \in \mathbb{N}_0$, then
\begin{equation*}
\mu^{(i, l)}(k) \in T^{(i)} \text{ for all } k \in \mathbb{N}_0 \text{ and } l \in \left\{n, \dots, n + l_\text{check}\right\}.
\end{equation*}
Thus, to ensure that $\mu^{(i, l)}$ remains inside the trust region for all $l \in  \{0, \dots, L^{(i)}\}$, it suffices to check $\Delta^{\mathcal{J}^{(i)}_\text{RB}}_r(\mu^{(i, l)}) \leq \epsilon_L^{(i)}$ every $l_\text{check}$-th iteration in the inner loop. Thus, if $l \mod l_\text{check} = 0$ hold, we call $l$ a \textit{checkpoint}. In addition to $\epsilon_L^{(i)}$, we define $\alpha_L^{(i)} > 0$, and demand $\alpha^{(i)}_0 \leq \alpha_L^{(i)}$ for all $i \in \mathbb{I}$. If $\mu^{(i+1)}$ is not accepted, $T^{(i)}$ is shrunk by setting $\epsilon_L^{(i+1)} := \beta_1 \epsilon_L^{(i)}$ for $\beta_1 \in (0,1)$ and $\alpha_L^{(i+1)} := \beta_2 \alpha_L^{(i)}$ for $\beta_2 \in (0,1)$. Otherwise, the trust region remains unchanged, i.e., $\epsilon_L^{(i+1)} := \epsilon_L^{(i)}$ and $\alpha_L^{(i+1)} := \alpha_L^{(i)}$. 

Notably, alternative algorithms have already been proposed, cf. \cite{keil2024relaxed}, that incorporate the idea of relaxing the trust region to allow larger errors in the inner loop, providing greater flexibility in model selection.

\begin{remark}
\label{rem:ML_more_accurate}
Beyond avoiding the computational cost of error estimation for the \Gls{ML} model, redefining the trust region has another significant advantage. Specifically, defining the trust region via $\Delta^{\mathcal{J}_\text{ML}}$ in a manner similar to \eqref{eq:orig_TR} could lead to unnecessary expansions of the reduced basis, resulting in computational overhead \cite{banholzer2020adaptiveprojectednewtonnonconforming, Keil_2021, keil2024relaxed}. 

Consider a scenario where the relative error of $\mathcal{J}_\text{ML}^{(i)}(\mu)$ exceeds the tolerance $\epsilon^{(i)}$, while the error for $\mathcal{J}_\text{RB}^{(i)}(\mu)$ remains below this threshold. If a trust region like $\tilde{T}^{(i)}$ were used, it would unnecessarily trigger a fallback to the main loop, leading to costly reconstruction of the reduced basis although the \Gls{RB} model itself remains adequate. Given the iterative retraining and volatile nature of $M^{(i,m)}_\text{ML}$, such cases are frequent but can typically be resolved by retraining $M^{(i,m)}_\text{ML}$, as long as the underlying \Gls{RB} space $V^{(i)}_\text{RB}$ remains sufficient.
\end{remark}

\begin{figure}[!b]
\centering
\includegraphics[width=0.95\textwidth]{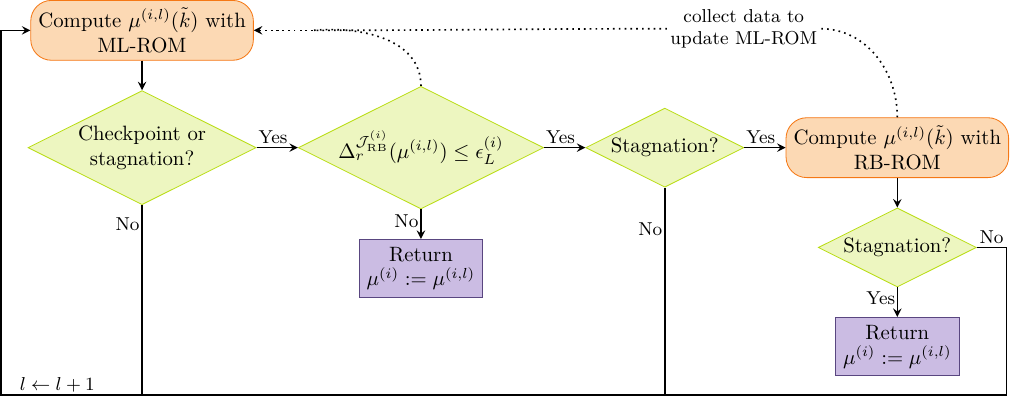}
\caption{\label{fig:flow_chart} Simplified flowchart illustrating the inner loop of the relaxed \Gls{TR}-\Gls{RB}-\Gls{ML} algorithm, as described in Algorithm \ref{algo:TR_opt}, after the warm-up phase.}
\end{figure}

With these considerations, we can now define the \Gls{TR}-\Gls{RB}-\Gls{ML} algorithm. Our goal is to build upon the approaches in \cite{qian2017certified, Keil_2021}, using $M^{(i)}_\text{RB}$ and $M^{(i, m)}_\text{ML}$ to compute $\mathcal{J}^{(i)}(\mu)$ and $\nabla_\mu \mathcal{J}^{(i)}(\mu)$, depending on the state of the \acrshort{BFGS} optimization procedure. We aim to rely on $M^{(i, m)}_\text{ML}$ as much as possible while using $M^{(i)}_\text{RB}$ only when necessary. Importantly, \Gls{ML} models will not be employed in the outer loop (Algorithm \ref{algo:TR_opt}). All surrogate evaluations for conditions \eqref{eq:EASDC}, \eqref{eq:suff_cond}, and \eqref{eq:nec_cond} will be handled by $M^{(i)}_\text{RB}$. As discussed in Remark \ref{rem:ML_more_accurate}, this is crucial for determining whether a re-adaptation of $M^{(i)}_\text{RB}$ is required. Therefore, the outer loop remains as defined in Algorithm \ref{algo:TR_opt}, while we modify how subproblems are solved in the inner loop (line \ref{line:compute} in Algorithm \ref{algo:TR_opt}) to return $\mu^{(i)}$.

\begin{algorithm2e}[!b]
\DontPrintSemicolon
\caption{Inner loop}\label{algo:inner_loop}

Let $\mu^{(i)}$ and the local surrogate models $M^{(i)}_\text{RB}$ and $M^{(i, 0)}_\text{ML}$ be given.\;
Set $l := 0$, $\mu^{(i,0)} := \mu^{(i)}$, $\epsilon^{(i)} > 0$, $\tilde{\mathcal{H}}^{(i, 0)} := \text{Id}_{\mathbb{R}^P}$, $d^{(i,0)} := -\nabla_{\mu} \mathcal J^{(i)}(\mu^{(i,0)})$, $\texttt{no\_progress} := \texttt{false}$ and $\texttt{use\_ML} := \texttt{true}$.\;
$\mathcal J^{(i)}(\mu^{(i,0)}), \nabla_\mu \mathcal J^{(i)}(\mu^{(i,0)}) \leftarrow M^{(i)}_\text{RB}\texttt{.eval\_output}[\mu^{(i,0)}]$\;
\While{\eqref{eq:subprob_term_crit_I} is not satisfied or $l = 0$}{
	\If{$l\,\mathrm{mod}\,l_\text{check} = 0$ or \texttt{no\_progress}}{
		\lIf{$\Delta^{\mathcal{J}^{(i)}_\text{RB}}_r(\mu^{(i,l)}) > \epsilon_L^{(i)}$}{go to line \ref{line:return}.\label{line:tr_check}}
	}			
	Set $k := 0$ and $\texttt{no\_progress} \leftarrow \texttt{false}$.\;
	\While{\eqref{eq:armijo} is not satisfied}{
		Get $\mu^{(i, l)}(k)$ by \eqref{eq:update_backtracking}. \; 											
		\If{$l \leq l_\text{warmup}$ or not $\texttt{use\_ML}$}{
			$\mathcal J^{(i)}(\mu^{(i,l)}(k)), \nabla_\mu \mathcal J^{(i)}(\mu^{(i,l)}(k)) \leftarrow M^{(i)}_\text{RB}\texttt{.eval\_output}[\mu^{(i,l)}(k)]$
		}\Else{
			$\mathcal J^{(i)}(\mu^{(i,l)}(k)), \nabla_\mu \mathcal J^{(i)}(\mu^{(i,l)}(k)) \leftarrow M^{(i,m)}_\text{ML}\texttt{.eval\_output}[\mu^{(i,l)}(k)]$
		}
		\If{not \eqref{eq:cutoff} and $l > 0$}{
			$\texttt{no\_progress} \leftarrow \texttt{true}$\;
			Go to line \ref{line:begin_ifelse}.\;
		}
		$k \leftarrow k + 1$
	}
	\If{\texttt{no\_progress} and not \texttt{use\_ML}}{\label{line:begin_ifelse}
		Go to line \ref{line:return}.
	} \ElseIf{\texttt{no\_progress}} {
		$\texttt{use\_ML} \leftarrow \texttt{false}$\;
		$d^{(i,l)} \leftarrow -\nabla_\mu \mathcal{J}^{(i)}_\text{RB}(\mu^{(i,l)})$ and $\tilde{\mathcal{H}}^{(i, l)} \leftarrow \text{Id}_{\mathbb{R}^P}$.\;
	} \Else {
		$\texttt{use\_ML} \leftarrow \texttt{true}$ and $\texttt{no\_progress} \leftarrow \texttt{false}$\;
		$\mu^{(i,l+1)} := \mu^{(i,l)}(k)$\;
		Get $\mathcal{H}^{(i, l+1)}$ and $d^{(i, l+1)}$ by \eqref{eq:set_d}.\;
		$M^{(i, m+1)}_\text{ML} \leftarrow M^{(i, m)}_\text{ML}\texttt{.train}$\;
		$l \leftarrow l + 1$ and $m \leftarrow m + 1$\;
	}
}
\Return{$\mu^{(i+1)} := \mu^{(i,l)}$} \label{line:return}
\end{algorithm2e}

During the first $l_\text{warmup} \in \mathbb{N}$ iterations of this inner loop, $M^{(i)}_\text{RB}$ is utilized to compute objective functional and gradient. This warm-up phase is necessary for accumulating a sufficient training dataset $\mathcal{M}^\text{train}$, which is essential for initializing valid \Gls{ML} models. Moreover, in the first iteration (beginning at $\mu^{(i)}$), the cutoff condition \eqref{eq:cutoff} is not enforced. As highlighted in \cite{yue2013accelerating}, this ensures that a valid \Gls{AGC} $\mu^{(i,1)}$ will be found by the backtracking procedure, which is crucial for preventing stagnation of Algorithm \ref{algo:TR_opt} away from an optimum. In the remaining iterations, $M^{(i, m)}_\text{ML}$, trained on data previously collected by the \Gls{RB} model, will be used.
In the case that for $\mu^{(i, l)}$ no sufficient $\mu^{(i, l)}(k)$ can be found, before reaching the minimum step size, i.e. \eqref{eq:cutoff} can not be not satisfied, the backtracking is stopped and it is checked that the last accepted parameter is still in the relaxed trust region, i.e. $\Delta^{\mathcal{J}^{(i)}_\text{RB}}_r(\mu^{(i,l)}) \leq \epsilon_L^{(i)}$. If this check fails, $\mu^{(i, l)}$ is returned as $\mu^{(i+1)}$ to the main loop. If the check succeeds, $\tilde{\mathcal{H}}^{(i,l)}$ is reset to $\text{Id}_{\mathbb{R}^P}$ and $d^{(i,l)}$ is updated to $-\nabla_\mu \mathcal{J}^{(i)}_\text{RB}(\mu^{(i,l)})$. The backtracking is then rerun, this time using the RB model $M^{(i)}_\text{RB}$, for all calculations. Note that, in general, the $M^{(i)}_\text{RB}$ is more accurate than $M^{(i, m)}_\text{ML}$, because the \Gls{ML} model depends on \Gls{RB} data, which introduces an additional model error. If this second attempt succeeds, the model $M^{(i, m)}_\text{ML}$ is retrained on the updated training sets that are derived by leveraging the RB model. Subsequent \acrshort{BFGS} iterations proceed as before, utilizing the updated ML model $M^{(i, m+1)}_\text{ML}$. However, if again no $\mu^{(i, l)}(k)$ can be found satisfying \eqref{eq:armijo} \-- \eqref{eq:TR_cond}, $\mu^{(i, l)}$ is also returned as $\mu^{(i+1)}$ and the \Gls{RB} space re-adapted.

As previously mentioned, we can ensure condition \eqref{eq:TR_cond} by periodically verifying that the relative error of the \Gls{RB} model remains below $\epsilon^{(i)}$. This approach circumvents the necessity to check \eqref{eq:TR_cond} for each $k \in \mathbb{N}_0$ and $l \in \{0, \dots, L^{(i)} \}$. Consequently, the relative error is evaluated every $l_\text{check}$-th iteration. If the error surpasses this threshold, the parameter $\mu^{(i, l)}$ is returned as $\mu^{(i+1)}$ to the main loop. The complete process of the inner loop computation is detailed in Algorithm \ref{algo:inner_loop}. Combined with the outer loop procedure this defines the \Gls{TR}-\Gls{RB}-\Gls{ML} algorithm. Figure \ref{fig:flow_chart} shows a simplified flow chart, illustrating the key steps of the inner loop.

\section{Numerical experiments}
\label{sec:num_exps}
In this section, we apply the proposed algorithm to an example problem aimed at optimizing parameters for managing temperature distribution in a building. We compare its performance with established optimization algorithms to assess its effectiveness. Our objective is to evaluate the strengths and limitations of the proposed method and identify areas where further improvements appear necessary.

\sidecaptionvpos{figure}{m}
\renewcommand{\sidecaptionrelwidth}{0.35}
\begin{SCfigure}[][b]
\centering
\includegraphics[width=0.60\textwidth]{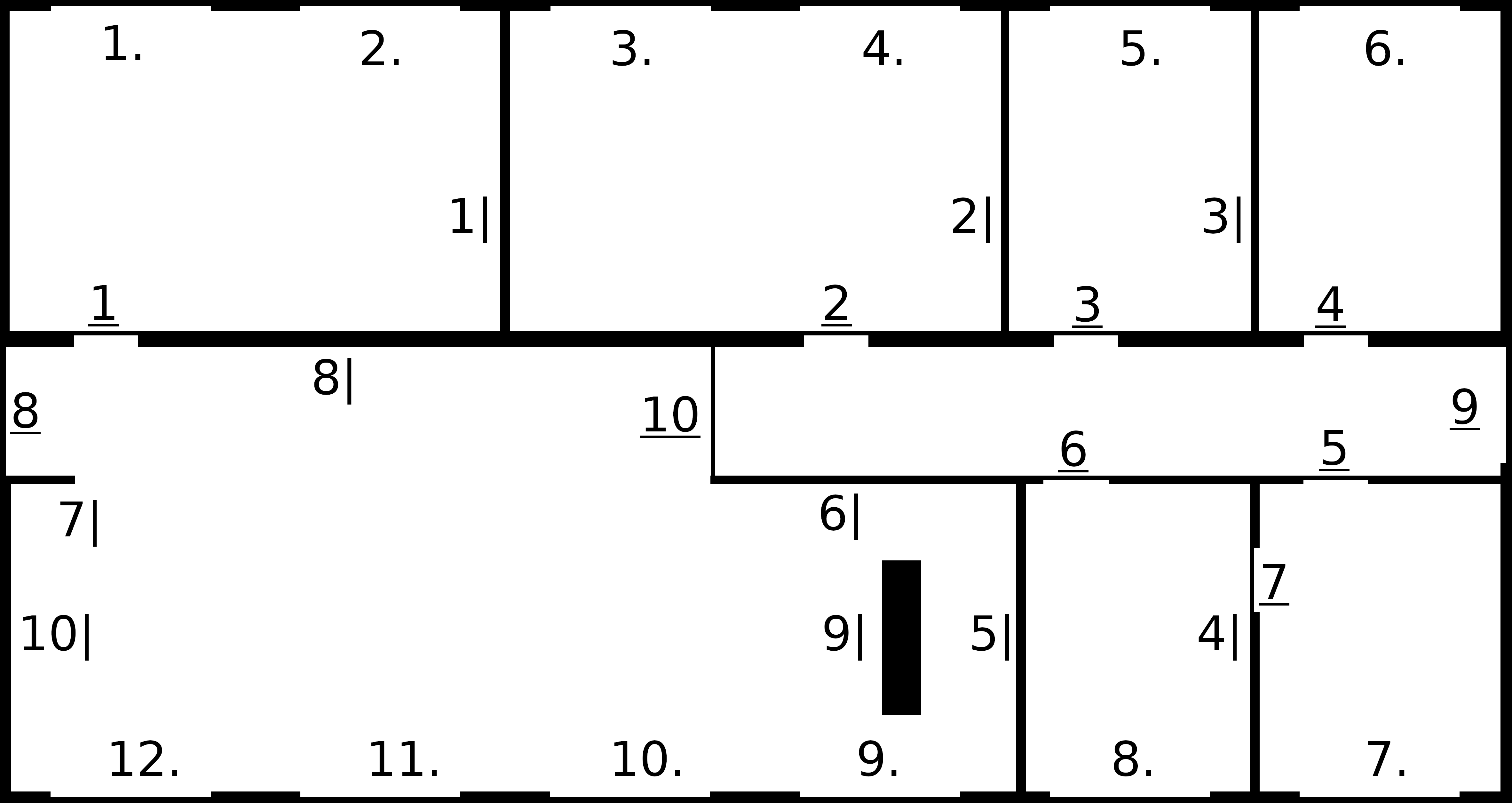}
\caption{\label{fig:building} From \cite{Keil_2021}: On the floor with domain $\Omega := (0,2) \times (0,1)$ heaters below the windows $i.$, doors $\underline{i}$ and walls $i|$ are places. Each has its own parameter $\mu_\ast$ describing the power of the heaters or respectively the heat conductivity of the components. Everywhere else the conductivity is fixed as $\mu_\text{air} = 0.5$.}
\end{SCfigure}

\subsection{Problem}
\label{ssec:num_exp_setup}
As a test case, we seek to determine the optimal heat conductivity values for structural elements (doors and walls) in a building floor $\Omega := (0,2) \times (0,1)$, as shown in Figure~\ref{fig:building}. The goal is to minimize the quadratic discrepancy between the actual temperature distribution $u(\mu)$ and a desired reference distribution $g_\text{ref}$. The temperature is governed by the instationary heat equation, where
\begin{equation*}
a(u,v; \mu) :=  \int_{\Omega} \kappa(\mu) \nabla u(x) \nabla v(x) \, dx, \quad
f(v; \mu) :=  \int_{\Omega} \eta(\mu) v(x) \, dx
\end{equation*}
for all $v \in V$. We impose homogeneous Dirichlet boundary conditions on $\partial \Omega$ and set the initial temperature distribution to $u^0(\mu) \equiv 0$. The parameters $\kappa(\mu)$ and $\eta(\mu)$ model the building components as linear combinations of indicator functions, specifying the locations and corresponding material properties. The diffusion coefficient $\kappa(\mu)$ governs the heat conductivity of walls and doors, while $\eta(\mu)$ in the source term represents the heating capacity, scaled by the function $b(t) := \min \{2t, 1\}$, which accounts for the time required for heaters to reach full power.

The outer walls and doors $(\mu_{10|}, \mu_{\underline{8}},\mu_{\underline{9}})$ have a prescribed heat conductivity of $5.0 \cdot 10^{-3}$, while the heating capacity of all heaters is fixed at $80$. The remaining walls and doors are assigned common conductivity values, denoted by $\mu_\text{walls}$ and $\mu_\text{doors}$, respectively. Our objective is to determine optimal values for these parameters within the range $[0.01, 0.1]$, i.e., $\mathcal{P} := [0.01, 0.1] \times [0.01, 0.1] \subset \mathbb{R}^2$. We define the function space $V := H^1_0(\Omega)$ and the inclusion operator $F: H^1_0(\Omega) \to L^2(\Omega)$, where $L^2(\Omega)$ is equipped with the inner product $\frac{10^4}{2} (\cdot, \cdot)_{L^2(\Omega)}$. A triangular spatial grid with $10,151$ nodes ($\dim V_h = 10,151$) is used, along with a temporal grid consisting of $10,001$ points ($K = 10,000$) with uniform time steps $\Delta t = 1/10$, forming the \acrlong{FOM}. The reference distribution $g_\text{ref} \in Q^\text{pr}_{\Delta t}(0,T; V_h)$ is obtained as the primal solution of the \acrshort{FOM} for $\hat{\mu} := (0.05, 0.05)$. Regularization is imposed through a quadratic penalty on deviations from $\hat{\mu}$, with a weight of $\lambda = 0.5 \cdot 10^{-3}$. Consequently, the objective functional is given by

\begin{equation*}
\mathcal{J}_h(\mu) = \frac{10^3}{2} \sum_{k = 1}^K \|u^k_h(\mu)  - g^k_\text{ref}\|^2_{L^2(\Omega)} + \frac{10^{-3}}{2}  \sum_{i = 1}^{P} \|\mu_i - \hat{\mu}_i\|^2_{\mathbb{R}^P}.
\end{equation*}
Since this functional is convex, it guarantees the existence of a unique minimizer at $\hat{\mu}$.

\subsection{Kernel methods}
\label{ssec:kernel_methods}

As outlined in Section \ref{sec:ML_surrogates}, \gls{ML} methods must satisfy certain criteria to be viable for optimization tasks. Selecting the optimal method involves balancing various properties. For instance, a \Gls{DNN} is generally unsuitable due to its over-parameterization, high data requirements, and lack of interpolation properties. Consequently, we employ a kernel-based model for our numerical experiments. Kernel methods are a mathematically well-studied field and widely applied in varying contexts including numerical mathematics, machine learning and many other areas. For a broader introduction, we refer to \cite{rasmussen2006gaussian, schaback2006kernel, santin2019kernel}.

The surrogate function $T : \mathcal{P} \rightarrow \mathbb{R}^q$ \eqref{eq:trained_func} is constructed as a vector-valued linear combination of kernel functions centered at distinct training inputs $\mu_1, \dots, \mu_{N_\text{train}}$ with corresponding coefficients $\alpha_1, \dots, \alpha_{N_\text{train}} \in \mathbb{R}^q$:
\begin{equation}
\label{eq:kernel_Dof}
T(\cdot) := \sum_{i = 1}^{N_\text{train}} \alpha_i K(\cdot, \mu_i) \in \text{span} \left\{ vK(\cdot, \mu_i)\,|\, i \in \{1, \dots, N_\text{train}\},\, v \in \mathbb{R}^q \right\} =: \mathcal{F}.
\end{equation}
Here, $q = (K+1)N^{(i)}_\text{RB}$, where $N^{(i)}_\text{RB}$ denotes the dimension of the \Gls{RB} space on which the \Gls{ML}-\Gls{ROM} is defined. The kernel function $K : \mathcal{X} \times \mathcal{P} \rightarrow \mathbb{R}^{q \times q}$ is symmetric, i.e., $K(\mu, \tilde{\mu}) = K(\tilde{\mu}, \mu)^T$ for all $\mu, \tilde{\mu} \in \mathcal{P}$. Given training data $\{(\mu_1, y_1), \dots, (\mu_N, y_{N_\text{train}})\}$, the goal is to determine optimal coefficients $\alpha_1, \dots, \alpha_{N_\text{train}}$ such that $T(\mu_i) = y_i$ for all $i \in \{1, \dots, N_\text{train}\}$. This corresponds to solving the linear system:
\begin{equation}
\label{eq:kernel_linear_problem}
A\alpha = y, \quad \text{where } \alpha := (\alpha_1, \dots, \alpha_{N_\text{train}} ), \quad y := (y_1, \dots, y_{N_\text{train}}),
\end{equation}
with the kernel matrix $A := (K(\mu_i, \mu_j))_{(i,j)} \in \mathbb{R}^{qN_\text{train} \times qN_\text{train}}$. If $K$ is strictly positive definite, this system has a unique solution. Training the model thus reduces to solving this linear system, typically via numerical methods such as QR decomposition. To improve numerical stability, $A$ is often regularized by adding $\eta \text{Id}_{\mathbb{R}^{qN_\text{train}}}$ for some $\eta > 0$, yielding:
\begin{equation}
\label{eq:kernel_linear_problem_reg}
(A + \eta \text{Id}_{\mathbb{R}^{qN_\text{train}}})\alpha = y.
\end{equation}
If $K(\cdot, \mu_i)$ is continuously differentiable for all $i \in \{1, \dots, N_\text{train}\}$, then for sufficiently small $\eta$, the kernel-based \Gls{ML} model satisfies requirement \eqref{eq:ml_interpol_error}. Furthermore, kernel methods inherently satisfy the initial conditions in Definition \ref{def:ML_surrogate}, i.e., ${\underline{u^0_\text{ML}}_n \equiv 0}$ and $d_{\mu_i} \underline{u^0_\text{ML}}_n \equiv 0$ for $1 \leq n \leq N_\text{RB}$, due to ${u^0_{\text{RB}} \equiv 0}$ and the uniqueness of \eqref{eq:kernel_linear_problem_reg}. 

The training of such kernel methods is computationally efficient as long as the dataset size remains moderate. However, scalability issues arise with large datasets. In such cases, alternative approaches like \Gls{VKOGA} \cite{santin2019kernel} can mitigate these challenges. In our case, we instead limit the training set to the most recent $N_\text{train} = 10$ evaluated parameters and their corresponding \Gls{RB} solutions. This ensures local accuracy near the current optimization trajectory while improving numerical stability, particularly when training data clusters within $\mathcal{P}$. It also allows for smaller regularization parameters, reducing interpolation error. For our experiments, we set $\eta = 10^{-12}$. Additionally, to further enhance numerical stability, training data is normalized to the range $[-1,1]^2$.

A crucial factor in approximation quality is the choice of the kernel function. For our application, radial basis function (RBF) kernels are particularly effective. We employ a Gaussian kernel, i.e.
$
K(\mu_i, \mu_j) := e^{-0.01 \|\mu_i-\mu_j\|_{\mathbb{R}^P}^2}\text{Id}_{\mathbb{R}^q},
$
which is continuously differentiable for all $\mu_j \in \mathcal{P}$, with derivative

\begin{equation*}
d_{\mu_i} K(\mu, \mu_j) = -2 \times 10^{-4} (\mu - \mu_j)_i e^{-0.01 \|\mu_i-\mu_j\|_{\mathbb{R}^P}^2}\text{Id}_{\mathbb{R}^q}.
\end{equation*}

\subsection{Methods used}
We conduct four separate optimization experiments, each employing a different algorithm:

\begin{enumerate}
    \item \textbf{\Gls{FOM}-Opt:} Optimization is performed using \acrshort{BFGS} directly on the \Gls{FOM}, without a \Gls{TR} scheme.
    
    \item \textbf{\Gls{TR}-\Gls{RB}-Opt:} Optimization follows a \Gls{TR} scheme with \Gls{RB}-\Gls{ROM}, as demonstrated by Qian et al. \cite{qian2017certified}.
    
    \item \textbf{Relaxed \Gls{TR}-\Gls{RB}-Opt:} The optimization algorithm described in Section \ref{sec:TR} is applied, but using only the \Gls{RB} model for parameter inference, without the \Gls{ML} model.
    
    \item \textbf{Relaxed \Gls{TR}-\Gls{RB}-\Gls{ML}-Opt:} Optimization follows the algorithm detailed in Section \ref{sec:TR}, incorporating the \Gls{ML} model.
\end{enumerate}

All algorithms are implemented in \texttt{Python}, utilizing \texttt{pyMOR} \cite{Milk_2016} for discretization and model order reduction, including \Gls{HaPOD}. The kernel method implementation also leverages functions from \Gls{VKOGA} \cite{santin2019kernel}. The source code is available at \cite{klein_2025_15063806}. All computations were performed on a custom-built PC equipped with an AMD Ryzen 7 3700X 8-Core Processor (16 threads) and 96 GB RAM. The continuity and coercivity constants, or their respective bounds, were computed analogously to \cite{qian2017certified} using the \textit{Riesz-eigenvalue approach} for $\gamma_d$ and the \textit{min-/max-theta method} for the bounds $\alpha_\text{LB}(\mu)$ and $\gamma_{a_{\mu_i}}^\text{UB}(\mu)$; see \cite{qian2017certified, haasdonk_min_theta} for details.

Across all experiments, the optimization termination tolerance is set to $\tau = 10^{-3}$. The backtracking step size decay factor is $\kappa_{\text{bt}} = 0.5$, with an initial step size of $\alpha^{(0)}_0 = 10^{-3}$. For all trust-region-based algorithms (Experiments 2 -- 4), the subproblem termination tolerance is set to $\tau_{\text{sub}} = 5 \cdot 10^{-4}$, while the initial relative error tolerance is $\epsilon^{(0)}_L = 0.1$, with a decay factor of $\beta_1 = 0.95$.  For relaxed trust-region variants (Experiments 3 and 4), the relative error is checked every 25 steps ($l_{\text{check}} = 25$), with a minimum step size of $\epsilon_{\text{cutoff}} = 10^{-6}$. The initial maximum size is $\alpha^{(0)}_L = 0.01$, with a decreasing factor $\beta_2 = 0.95$. To generate training data and reinitialize the \Gls{ML}-\Gls{ROM} upon expanding the \Gls{RB} space, we set $l_{\text{warmup}} = 3$. 

Since optimization behavior depends on the initial guess $\mu^{(0)}$, we conduct ten independent runs per experiment, with $\mu^{(0)}$ uniformly sampled from $\mathcal{P}$. These initial values remain consistent across all experiments to ensure comparability.

\begin{table}[b]
\setlength{\tabcolsep}{5pt}
\renewcommand{\arraystretch}{1.0} 
\centering
\begin{tabular}{l|ccccc}
\toprule
Algorithm & Time $[s]$ & Speed-up & FOM eval. & RB-ROM eval. & ML-ROM eval. \\
\midrule
\Gls{FOM}-Opt & 2550.20 & -- & 68.40 & -- & -- \\
\Gls{TR}-\Gls{RB}-Opt & 911.65 & 2.80 & 7.50 & 154.60 & -- \\
Relaxed \Gls{TR}-\Gls{RB}-Opt & 406.42 & 6.27 & 4.40 & 90.70 & -- \\
Relaxed \Gls{TR}-\Gls{RB}-\Gls{ML}-Opt & 304.50 & 8.38 & 5.10 & 30.90 & 108.00 \\
\bottomrule
\end{tabular}
\caption{\label{tab:runtimes_and_num_evals} Total runtime and number of evaluations for \Gls{FOM}, \Gls{RB}-\Gls{ROM}, and \Gls{ML}-\Gls{ROM} across selected algorithms, averaged over ten optimization runs with different initial guesses.}
\end{table}

\begin{figure}[t]
\centering
\includegraphics[width=0.95\textwidth]{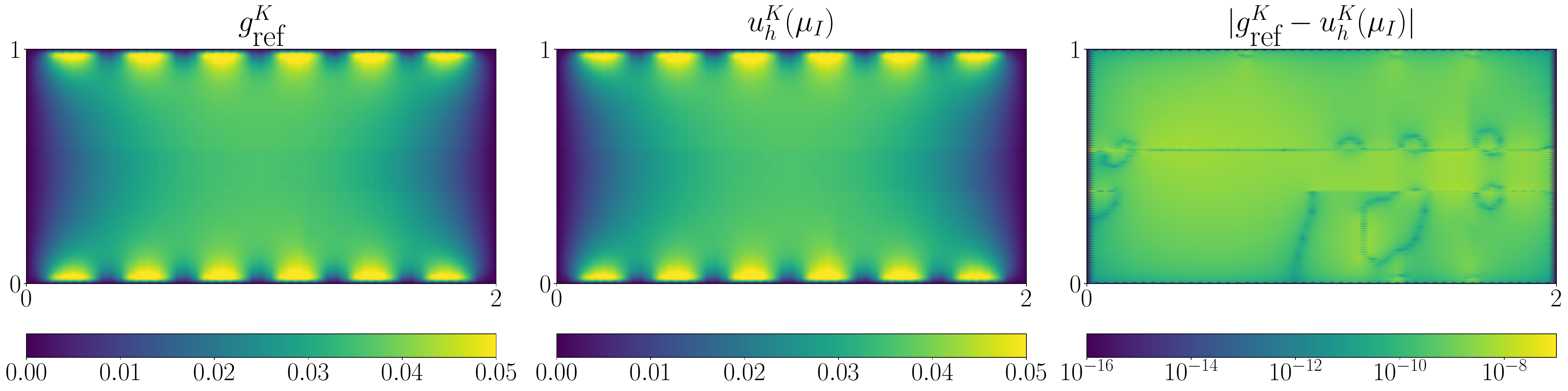}
\caption{\label{fig:state_compare} 
Comparison of the desired state $g_{\text{ref}}$ (left), the final state obtained using the relaxed 
\Gls{TR}-\Gls{RB}-\Gls{ML} algorithm with initial guess $\mu^{(0)} = (0.026, 0.020)$ (middle), 
and their absolute differences (right) at the final time step ($K = 10000$).}
\end{figure}

\subsection{Results}
All optimization runs successfully converge within the specified tolerance to the global minimum at $\bar{\mu} = (0.05, 0.05)$. An example is illustrated in Figure \ref{fig:state_compare}. Table \ref{tab:runtimes_and_num_evals} presents the total runtimes and the number of model evaluations, averaged over ten optimization runs with varying initial guesses. All algorithms outperform the \Gls{FOM} variant in terms of computational efficiency. Notably, the relaxed \Gls{TR}-\Gls{RB}-\Gls{ML} approach achieves the most significant reduction in computation time, with a speed-up factor of 8.38. This result demonstrates that augmenting \Gls{RB}-\gls{ROM} with \Gls{ML}-\Gls{ROM} can be a viable strategy for reducing computational costs.

This becomes even more evident when examining the evaluation times for individual queries in greater detail. Figure \ref{fig:primal_run_bars} illustrates the evaluation times for the optimization using the relaxed \Gls{TR}-\Gls{RB} method, starting at $\mu^{(0)} = (0.026, 0.020)$. The figure shows the time required to compute primal solutions and error estimates of the \Gls{RB}-\Gls{ROM} for each queried parameter. Most evaluations are carried out by \Gls{ML}-\glspl{ROM}, while \Gls{RB}-\glspl{ROM} are primarily used at the beginning of each subproblem iteration for collecting training data and towards the end of the optimization process, where they are used more frequently, due to the more accurate approximations. As expected, \Gls{ML}-\Gls{ROM} evaluations are significantly faster, up to a factor of $10^2$, compared to \Gls{RB}-\Gls{ROM} evaluations. Moreover, the effective online speed-up is even greater since gradients can be computed directly via the chain rule, whereas \Gls{FOM} and \Gls{RB}-\Gls{ROM} require additional calculations obtaining the adjoint solution.

Despite the impressive speed-up in individual parameter inferences, the overall speed-up is less pronounced, as shown in Table \ref{tab:runtimes_and_num_evals}. This discrepancy arises because the optimization process includes several additional computationally intensive steps that remain unoptimized and do not benefit from \Gls{ML}-\Gls{ROM} acceleration.

\begin{figure}[t]
\centering
\includegraphics[width=\textwidth]{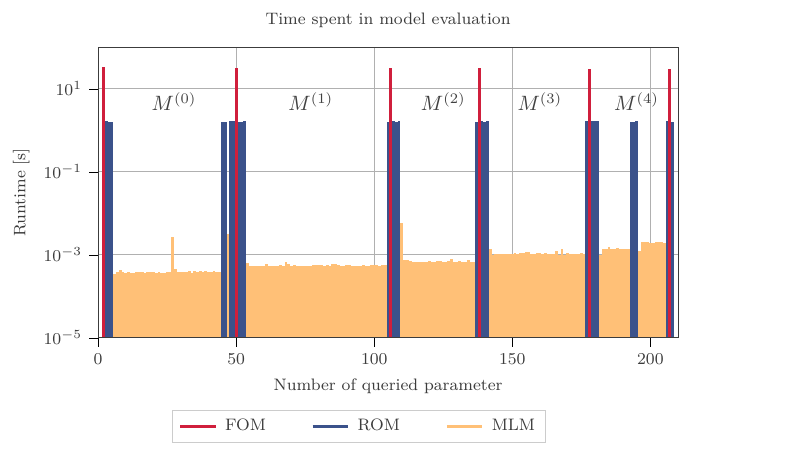}
\caption{\label{fig:primal_run_bars} Evaluation times of the primal models for each queried parameter during the run performed by the relaxed \Gls{TR}-\Gls{RB}-\Gls{ML} algorithm, starting at $\mu^{(0)} = (0.026, 0.020)$. The colors indicate the model used: \Gls{FOM} (red), \Gls{RB}-\Gls{ROM} (blue), and \Gls{ML}-\Gls{ROM} (orange). Additionally, the scope of the \Gls{RB}-\Gls{ROM}s $M^{(i)}$ is shown.}
\end{figure}

The most time-consuming steps in the optimization process are \Gls{FOM} evaluations, required for constructing \Gls{RB} spaces, and \Gls{RB}-\Gls{ROM} evaluations, necessary for generating training data. In the relaxed \Gls{TR}-\Gls{RB}-\Gls{ML} optimization, \Gls{ML}-\Gls{ROM} evaluations contribute, on average, less than $0.04\%$ of the total runtime, while \Gls{FOM} evaluations account for $54.44\%$, \Gls{RB}-\Gls{ROM} evaluations for $12.45\%$, and \Gls{RB} space construction for $12.69\%$. Another costly step is the initialization of the optimization setup, particularly the construction of the \Gls{FOM}. These findings align with previous studies \cite{banholzer2020adaptiveprojectednewtonnonconforming, Keil_2021}, which highlight the extension of the \Gls{RB} space as a critical performance bottleneck in \Gls{TR} approaches.

From Table \ref{tab:runtimes_and_num_evals}, we observe that incorporating \Gls{ML}-\Gls{ROM} can sometimes negatively impact performance. Optimization runs using \Gls{ML} methods generally require more \Gls{FOM} evaluations and, consequently, more \Gls{RB} space extensions compared to the relaxed \Gls{TR}-\Gls{RB} method without \Gls{ML}-\Gls{ROM}. This occurs because \Gls{ML}-\glspl{ROM} typically exhibit higher errors than \Gls{RB}-\glspl{ROM}, potentially leading to stagnation away from the optimum. Such stagnation triggers the check that $\mu^{(i, l)}$ is an element of $\tilde{T}^{(i)}$ (Line \ref{line:tr_check} in Algorithm \ref{algo:inner_loop}), potentially resulting in more frequent fallbacks to the main loop. This also explains why total parameter evaluations increase when \Gls{ML}-\Gls{ROM} is incorporated, as seen in Table \ref{tab:runtimes_and_num_evals}.

However, this also suggests that the observed speed-up cannot be entirely attributed to the incorporation of \Gls{ML}-\glspl{ROM}. A major contributing factor to the reduced computation time for the relaxed \Gls{TR}-\Gls{RB}-\Gls{ML} optimization is the reduction in \Gls{FOM} evaluations due to the introduction of the relaxed trust region. This becomes particularly evident when comparing the speed-up factors of \Gls{TR}-\Gls{RB}-Opt (2.80) and relaxed \Gls{TR}-\Gls{RB}-Opt (6.27). A likely explanation is that the error estimators introduced in Section \ref{sssec:a_post_err_est} tend to overestimate the actual error of \Gls{RB} solutions, imposing unnecessarily strict restrictions for convergence and prompting premature \Gls{RB} space reconstructions. An additional effect that further reduces computation time is the reduction in \gls{RB} evaluations required. This occurs because it is no longer demanded that $\mu^{(i)}(k)$ remains within the trust region, allowing the backtracking routine to more quickly return valid steps. These observations underscore that improving \gls{TR}-\gls{RB} methods hinges not only on developing better surrogate models but also on finding higher-quality error estimators and more refined conditions to manage the costly re-instantiation of the surrogate models.

Finally, we discuss the characteristics of \Gls{ML}-\Gls{ROM} in an optimization context. We also tested alternative \Gls{ML} methods, including deep neural networks and 2L-VKOGA \cite{wenzel2023datadrivenkerneldesignsoptimized}, but did not achieve satisfactory results. This highlights issues that need to be addressed to make machine learning a more reliable foundation for surrogate models. Our experiments reveal that the requirements on the \Gls{ML} algorithm outlined in Section \ref{ssec:ML_general_considerations} are quite restrictive for practical applications. In particular, the constraint that \Gls{ML} methods must train efficiently on limited data prevents the use of more complex models. Combined with the need to guarantee a certain level of accuracy for effective optimization, these constraints limit both the applicability of \Gls{ML}-based optimization and the achievable computational savings.

This issue is exacerbated during optimization, as the parameters for which the \Gls{ML}-\Gls{ROM} is evaluated often lie \textit{outside} of the training data distribution. Consequently, the approximation quality depends more on the hypothesis set and hyperparameters of the model than on the training data itself. This necessitates that the model not only \textit{generalizes} well but also accurately captures the local solution manifold.

However, ensuring adequate generalization requires careful selection of the hypothesis set, balancing model complexity against training time and requirements on the training data. This involves choosing appropriate hyperparameters, kernel functions, and regularization constants, which must be set a priori. Hyperparameter tuning is already challenging, as it requires empirical validation, and becomes even more difficult during optimization, where the training data distribution and required accuracy evolve dynamically. Near the optimum, training data tend to cluster due to frequent backtracking steps, increasing the required accuracy and demanding different kernel parameterizations than those suitable at the start of the optimization. Attempts to adapt hyperparameters dynamically during optimization yielded unpredictable results and did not significantly improve overall performance.
 
\section{Conclusion and Outlook}

In this work, we have demonstrated that kernel methods can be effectively applied to optimize quadratic cost functionals constrained by parametrized parabolic \glspl{PDE} within a trust-region reduced-basis framework, offering a viable approach to reducing computational costs. By learning solutions to the \Gls{PDE} constraints projected onto reduced spaces, the proposed \Gls{ML}-\Gls{ROM} can be certified, while the application of the chain rule enables direct access to the gradient of the cost functional. However, the derived error estimators remain computationally inefficient  relative to parameter inference by the \Gls{ML} surrogate. The proposed algorithm addresses this issue by integrating these aspects into a trust-region scheme while circumventing direct reliance on the error estimator. This is achieved by relaxing the \Gls{TR} constraint, ensuring the validity of the surrogates while maintaining computational efficiency. Additionally, the hierarchical model-calling scheme allows the algorithm to dynamically adapt to the current region of interest along the optimization trajectory.

However, as demonstrated in the example of temperature distribution matching, further advancements are necessary to establish \Gls{ML}-based surrogates as a practical choice, particularly for more complex problems. Beyond the need for sharper and more efficient error estimators and greater robustness to erroneous surrogate models, improvements in the \Gls{ML} models themselves are required. The key challenge is to enable the use of more sophisticated models while maintaining computational feasibility. Future efforts should focus on approaches that incorporate additional information about the \Gls{PDE} constraints alongside training data, such as semi-supervised machine learning methods. 

Of particular interest are recent developments in \glspl{PINN} and their variants \cite{raissi2019physics, yu2018deep, tanyu2023deep}. These models explicitly integrate \Gls{PDE} constraints into their architecture and, in theory, can produce meaningful global approximations even with limited training data. Incorporating \glspl{PINN} into the existing framework could therefore be a promising avenue for future research, advancing the integration of \Gls{ML} methods into parametrized parabolic \Gls{PDE}-constrained optimization tasks.

\section*{Statements and declarations}
\bmhead{Funding} The authors acknowledges funding by the Deutsche Forschungsgemeinschaft (DFG, German Research Foundation) under Germany's Excellence Strategy EXC 2044–390685587, Mathematics Münster: Dynamics–Geometry–Structure.
\bmhead{Competing interests} The authors have no relevant financial or non-financial interests to disclose.
\bmhead{Acknowledgements} Benedikt Klein thanks Hendrik Kleikamp from the University of Münster for the long and fruitful exchange.

\bibliography{./source.bib}

\begin{appendices}

\section{\Gls{ML} a posteriori error estimator}
\label{sec:appendix_proofs}

\subsection{Proof of Lemma \ref{lem:error_bound_u_ML}}

\begin{proof}
The high-fidelity solution map $\mu \mapsto u_h(\mu) \in Q^\text{pr}_{\Delta t}(0,T; V_h)$ is Fr\'{e}chet differentiable. For all $k \in \mathbb{K}$ and $i \in \{1, \dots, P \}$. $d_{\mu_i}u^k_h(\mu)$ can be obtained as the unique solution of
\begin{equation}
\label{eq:eq_du_h}
d_{\mu_i}r^k_\text{pr}(u^k_h, v;\mu) = a(d_{\mu_i}u^k_h(\mu), v; \mu) + \frac{1}{\Delta t }(d_{\mu_i}u_h^k(\mu)- d_{\mu_i}u_h^{k-1}(\mu), v)_{L^2(\Omega)},
\end{equation}
for all $v \in V_h$ with initial condition $d_{\mu_i}u^0_h(\mu) = 0$, cf. \cite[Section 1.6.1]{hinze2009optimization}. Adding \eqref{eq:residuum_dmu_ML} and \eqref{eq:eq_du_h} gives
\begin{equation*}
a(d_{\mu_i}e^k, v; \mu) + \frac{1}{\Delta t }(d_{\mu_i}e^k(\mu)- d_{\mu_i}e^{k-1}(\mu), v)_{L^2(\Omega)} = \mathcal{R}^k_\text{ML}(v; \mu) + d_{\mu_i}r^k_\text{pr}(e^k(\mu), v;\mu),
\end{equation*}
where $e^k(\mu):= u^k_h(\mu) - u^k_\text{ML}(\mu)$ and $d_{\mu_i}e^k(\mu):= d_{\mu_i}u^k_h(\mu) - d_{\mu_i}u^k_\text{ML}(\mu)$. Setting $v := d_{\mu_i}e^k(\mu)$ and using 
\begin{equation*}
d_{\mu_i}r^k_\text{pr}(e^k(\mu), v;\mu) = d_{\mu_i} a(-e^k(\mu), v;\mu) \leq \gamma_{a_{\mu_i}}(\mu)\|e^k(\mu)\|_{V_h}\|v\|_{V_h},
\end{equation*} 
for all $v \in V_h$ gives 
\begin{gather*}
(d_{\mu_i}e^k(\mu), d_{\mu_i}e^k(\mu))_{L^2(\Omega)}  + \Delta t a(d_{\mu_i}e^k, d_{\mu_i}e^k(\mu); \mu) \\
\leq (d_{\mu_i}e^{k-1}(\mu), d_{\mu_i}e^k(\mu))_{L^2(\Omega)} + \Delta t \left(\| \mathcal{R}^k_\text{ML}(\cdot\,; \mu)\|_{V'_h} + \gamma_{a_{\mu_i}}(\mu)\|e^k(\mu)\|_{V_h}\right)\|d_{\mu_i}e^k(\mu)\|_{V_h}
\end{gather*}
Analogue to the proof of Proposition 4.1 in \cite{grepl2005posteriori} remark firstly
\begin{gather*}
(d_{\mu_i}e^{k-1}(\mu), d_{\mu_i}e^k(\mu))_{L^2(\Omega)} \\
\leq \frac{1}{2}\left(
(d_{\mu_i}e^{k-1}(\mu), d_{\mu_i}e^{k-1}(\mu))_{L^2(\Omega)} + 
(d_{\mu_i}e^{k}(\mu), d_{\mu_i}e^k(\mu))_{L^2(\Omega)}
\right)
\end{gather*}
and secondly 
\begin{gather*}
\left(\| \mathcal{R}_\text{ML}(\cdot\,; \mu)\|_{V'_h} + \gamma_{a_{\mu_i}}(\mu)\|e^k(\mu)\|_{V_h}\right) \|d_{\mu_i}e^k(\mu)\|_{V_h}\\
\leq
\frac{1}{2 \alpha_\text{LB}(\mu)} \left(\| \mathcal{R}^k_\text{ML}(\cdot\,; \mu)\|_{V'_h} + \gamma_{a_{\mu_i}}(\mu)\|e^k(\mu)\|_{V_h}\right)^2 + \frac{1}{2} a(d_{\mu_i}e^k(\mu), d_{\mu_i}e^k(\mu); \mu),
\end{gather*}
resulting in 
\begin{gather*}
(d_{\mu_i}e^k(\mu), d_{\mu_i}e^k(\mu))_{L^2(\Omega)} - (d_{\mu_i}e^{k-1}(\mu), d_{\mu_i}e^{k-1}(\mu))_{L^2(\Omega)} + \Delta t a(d_{\mu_i}e^k(\mu), d_{\mu_i}e^k(\mu); \mu) \\
\leq \frac{\Delta t}{\alpha_\text{LB}(\mu)} \left(\| \mathcal{R}^k_\text{ML}(\cdot\,; \mu)\|_{V'_h}^2 + \gamma^2_{a_{\mu_i}}(\mu)\|e^k(\mu)\|^2_{V_h} + 2\gamma_{a_{\mu_i}}(\mu)\|\mathcal{R}^k_\text{ML}(\cdot\,; \mu)\|_{V'_h}\|e^k(\mu)\|_{V_h}\right)
\end{gather*}
Taking the sum over $k \in \mathbb{K}$ and using $d_{\mu_i}e^0(\mu) = 0$ results in
\begin{align*}
\alpha_\text{LB}(\mu) \mathcal{S}^2(d_{\mu_i}e(\mu)) \leq & \alpha^{-1}_\text{LB}(\mu) \mathcal{T}^2(\mathcal{R}_\text{ML}(\cdot\,; \mu)) \\
+ & \alpha^{-1}_\text{LB}(\mu) \gamma^2_{a_{\mu_i}}(\mu) \mathcal{S}^2(e(\mu)) + 2\alpha^{-1}_\text{LB}(\mu) \gamma_{a_{\mu_i}}(\mu)\mathcal{T}(\mathcal{R}_\text{ML}(\cdot\,; \mu))\mathcal{S}(e(\mu))
\end{align*}
with $e(\mu) := u_h(\mu) - u_\text{ML}(\mu)$. This proves the claim.
\end{proof}

\subsection{Proof of Theorem \ref{thm:error_est_ML}}
\begin{proof}
For \eqref{eq:J_ML} consider the first part of the proof to Theorem 9 in \cite{qian2017certified}, obtaining 
\begin{equation*}
e^{\mathcal{J}}(\mu) :=  \mathcal{J}_h(\mu) - \mathcal{J}_\text{ML}(\mu) = \Delta t \sum_{k = 1}^K \left[2d(u^k_\text{ML}(\mu), e^k(\mu)) + l^k(e^k(\mu)) + d(e^k(\mu),e^k(\mu))\right]
\end{equation*}
for $e^k(\mu) := u^k_h(\mu) - u^k_\text{ML}(\mu)$. Substituting $l^k(e^k(\mu)) = d(g^k_\text{ref}, e^k(\mu))$ results in
\begin{equation*}
|e^{\mathcal{J}}(\mu)| \leq \gamma_d  \Delta t \sum_{k = 1}^K \left[2 \|u^k_\text{ML}(\mu)\|_{V_h}\|e^k(\mu)\|_{V_h} + \|g^k_\text{ref}\|_{V_h}\|e^k(\mu)\|_{V_h} + \|e^k(\mu)\|^2_{V_h}\right].
\end{equation*}
Applying Cauchy-Schwarz inequality and using \eqref{eq:u_ML_error_bound} proves the claim. To prove the second inequality \eqref{eq:nabla_J_ML} denote
\begin{align*}
e^{d_{\mu_i} \mathcal{J}}(\mu) :=& d_{\mu_i} \mathcal{J}_h(\mu) - d_{\mu_i} \mathcal{J}_\text{ML}(\mu) \\
=& \Delta t \sum_{k = 1}^K \left[
\underbrace{2d(u^k_h(\mu), d_{\mu_i} u^k_h(\mu)) - 2d(u^k_\text{ML}(\mu), d_{\mu_i} u^k_\text{ML}(\mu))}_{=: (\ast)} + d(g^k_\text{ref}, d_{\mu_i}e^k(\mu))
\right].
\end{align*}
Adding and substracting 
\begin{equation*}
\Delta t \sum_{k = 1}^K \left[ 2d(u^k_h(\mu), d_{\mu_i} u^k_\text{ML}(\mu)) + 2d(u^k_\text{ML}(\mu), d_{\mu_i}e^k(\mu)) \right]
\end{equation*}
to $(\ast)$ gives 
\begin{align*}
(\ast) = \Delta t \sum_{k = 1}^K & \left[ 
2d(u^k_h(\mu), d_{\mu_i} e^k(\mu)) + 2d(e^k(\mu), d_{\mu_i} u^k_\text{ML}(\mu)) \right]\\
+ \Delta t \sum_{k = 1}^K & \left[ 2d(u^k_\text{ML}(\mu), d_{\mu_i}e^k(\mu)) - 2d(u^k_\text{ML}(\mu), d_{\mu_i}e^k(\mu)) \right] \\
= \Delta t \sum_{k = 1}^K & \left[ 2d(e^k(\mu), d_{\mu_i} e^k(\mu)) + 2d(e^k(\mu), d_{\mu_i} u^k_\text{ML}(\mu)) + 2d(u^k_\text{ML}(\mu), d_{\mu_i}e^k(\mu))\right].
\end{align*}
Resulting in 
\begin{gather*}
\begin{aligned}
|e^{d_{\mu_i}\mathcal{J}}(\mu)| \leq \gamma_d\Delta t \sum_{k = 1}^K & \left[2\|e^k(\mu)\|_{V_h}\|d_{\mu_i}e^k(\mu)\|_{V_h} + 2\|e^k(\mu)\|_{V_h}\|d_{\mu_i}u^k_\text{ML}(\mu)\|_{V_h}\right]  \\
+ \gamma_d\Delta t \sum_{k = 1}^K & \left[2\|u^k_\text{ML}(\mu)\|_{V_h}\|d_{\mu_i}e^k(\mu)\|_{V_h}
+ \|g^k_\text{ref}\|_{V_h}\|d_{\mu_i}e^k(\mu)\|_{V_h}\right].
\end{aligned}
\end{gather*}
The claim follows with Lemma \ref{lem:error_bound_u_ML}.
\end{proof}
\end{appendices}

\end{document}